\documentclass{amsart}
\usepackage{amscd, amssymb, amsmath, amsthm}
\usepackage{amsfonts,enumerate, latexsym}
\usepackage{graphics,graphicx,psfrag, mathrsfs}
\usepackage[colorlinks=true]{hyperref}

\newtheorem{theorem}{Theorem}[section]
\newtheorem{lemma}[theorem]{Lemma}
\newtheorem{prop}[theorem]{Proposition}
\newtheorem{cor}[theorem]{Corollary}
\newtheorem{conj}[theorem]{Conjecture}

\theoremstyle{definition}
\newtheorem{definition}[theorem]{Definition}

\theoremstyle{remark}
\newtheorem*{remark}{Remark}

\newcommand{\RR}{{\mathbf R}}
\newcommand{\QQ}{{\mathbf Q}}

\newcommand{\ZZ}{{\mathbf Z}}

\newcommand{\Sol}{{\mathrm{Sol}}}

\newcommand{\Hom}{{\mathrm{Hom}}}

\newcommand{\ud}{{\mathrm{d}}}

\newcommand{\CAT}{{\mathrm{CAT}}}

\newcommand{\tori}{{\mathcal{T}}}
\newcommand{\pieces}{{\mathcal{J}}}
\newcommand{\cfggraph}{{\Lambda}}

\newcommand{\vrtx}{{\mathrm{Ver}}}
\newcommand{\edge}{{\mathrm{Edg}}}
\newcommand{\lk}{{\mathrm{lk}}}

\newcommand{\cuts}{{\mathcal{Z}}}
\newcommand{\binds}{{\mathcal{W}}}
\newcommand{\pants}{{\mathcal{P}}}
\newcommand{\hypps}{{\mathcal{H}}}
\newcommand{\cutgraph}{{\Upsilon}}

\title[Virtual cubulation of NPC graph manifolds]{Virtual cubulation of nonpositively curved graph manifolds}

\author[Yi Liu]{%
        Yi Liu} 
\address{%
    Mathematics 253-27\\
    California Institute of Technology\\
    Pasadena, CA 91125} 
\email{%
    yliumath@caltech.edu}  

\subjclass[2010]{Primary 57M05; Secondary 20F65}

\date{%
 \today}

\begin{document}

\begin{abstract}
	In this paper, we show that an aspherical compact graph manifold is 
	nonpositively curved if and only if its fundamental group virtually embeds
	into a right-angled Artin group. As a consequence, nonpositively curved
	graph manifolds have linear fundamental groups.
\end{abstract}

\maketitle

\section{Introduction}

	In this paper, we study virtual properties
	(i.e.~properties that hold for some finite covering)
	of nonpositively curved graph manifolds.
	Our main result is the following, (cf. Section \ref{Sec-Prelim}
	for terminology):
	
	\begin{theorem}\label{main-npcGM} 
		Let $M$ be an aspherical compact graph manifold. 
		Then the following statements are equivalent:
		\begin{enumerate}
			\item $M$ is nonpositively curved;
			\item $M$ is virtually homotopy equivalent to a special cube complex;
			\item $\pi_1(M)$ virtually embeds into a finitely generated right-angled Artin group;
			\item $\pi_1(M)$ is virtually RFRS.
		\end{enumerate}
	\end{theorem}
	
	By the RFRS criterion of Ian Agol for virtual fibering (\cite{Ag}, cf.~Subsection \ref{Subsec-RFRS}),
	Theorem \ref{main-npcGM} strengthens a result of Pavel Svetlov that nonpositively curved graph manifolds 
	virtually fibers over the circle, (\cite{Sv}, cf.~Subsection \ref{Subsec-graphMfd}). We have also
	learned recently that Piotr Przytycki and Daniel Wise have an independent proof 
	of the special cubulation for the bounded case using criteria of separability and 
	double-coset separability of surface subgroups, (\cite{PW-graph}).
	
	It has also been conjectured that $3$-manifold groups are linear. Because
	finitely generated right-angled Artin groups
	are known to be linear (cf.~Subsection \ref{Subsec-RAAG}), and because linearity is a commensurability-class property,
	an interest of Theorem \ref{main-npcGM} is the following: 
	
	\begin{cor} 
		The fundamental group of any nonpositively curved compact graph manifold 
		is linear.
	\end{cor}
	
	Theorem \ref{main-npcGM} is motivated by some recent progress in the
	Virtual Fibering Conjecture (VFC). Back to the 1980s, William Thurston
	conjectured that every hyperbolic $3$-manifold virtually fibers over the circle,
	(\cite{Th-VFC}).
	Two years ago, Daniel Wise announced a proof of this conjecture
	for hyperbolic $3$-manifolds that contain essentially embedded 
	geometrically finite closed subsurfaces, using quasiconvex hierarchies
	(\cite{Wi}). In fact, he showed that every such manifold
	is virtually homotopy equivalent to a finite-dimensional
	special cube complex. Hence the fundamental group
	of any such manifold virtually embeds into a right-angled Artin group. 
	This implies virtual fibering by Agol's criterion. 
	Thus the VFC for hyperbolic $3$-manifolds is implied by
	the Virtual Haken Conjecture (VHC) that every closed $3$-manifold
	virtually contains an (embedded) incompressible subsurface.
	In more general context, the statement of the
	VFC is known to be false for many graph manifolds, (cf.~Subsection
	\ref{Subsec-graphMfd}), and for any
	closed Seifert-fibered space with nontrivial Euler class of the fiber. 
	Nevertheless, it has been conjectured that
	the VFC holds for nonpositively curved $3$-manifolds,
	(cf.~\cite[Conjecture 9.1]{Ag}). Based on Wise's work and Theorem
	\ref{main-npcGM}, a reasonable specification of this conjecture might be the
	following\footnote{There has been significant progress on virtual specialization
	of $3$-manifold groups since the announcement of this paper in the fall of 2011.
	The remaining cases	of Conjecture \ref{npcVSp} are verified by
	\cite{Agol-VHC,PW-mixed}.
	The introduction is now a little behind the times
	but we prefer keeping it as a historical record.}:
	
	\begin{conj}\label{npcVSp}
		An aspherical compact $3$-manifold is nonpositively curved
		if and only if its fundamental group virtually embeds into a right-angled 
		Artin group.
	\end{conj}
	
	We explain our strategy to prove Theorem \ref{main-npcGM}. 
	It suffices to consider the nontrivial case, namely, where
	the graph manifold is not itself Seifert fibered. From known results,
	(cf.~Theorems \ref{Thm-Sv}, \ref{Thm-Ag}), all the statements of Theorem \ref{main-npcGM} imply
	virtual fibering. By a trick of doubling the $3$-manifold along boundary, and by 
	verifying the Anosov-mapping-torus case directly, we shall focus on the
	case that $M$ is the mapping torus $M_\sigma$ of 
	a multitwist $\sigma$ on an oriented closed surface $F$, (Definition \ref{multitwist}). Our plan is to show (1)
	$\Rightarrow$ (2) $\Rightarrow$ (3) $\Rightarrow$ (4) $\Rightarrow$ (1), 
	where only the first and the last implications require our effort.
	
	To show virtual RFRS-ness implies nonpositive curving, we introduce the current equations associated to the multitwist
	$\sigma$. A nondegenerate symmetric solution 
	of this system of equations is virtually
	ensured by the RFRS condition. 
	Such a solution can be pertubed into a nondegenerate symmetric solution of
	certain structural equations regarding to nonpositively curved metrics,
	called the BKN equations.
	This will provide a nonpositively curved metric on $M_\sigma$.
	To show that nonpositive curving implies a virtual special cubulation,
	we apply a technique used by Svetlov to find a virtually embedded surface,
	which virtually guarantees a nondegenerate symmetric solution to the current equations. With such
	a solution in hand, we shall find a finite collection of first cohomology classes,
	whose duals are supposably represented by 
	subsurfaces of $M_\sigma$ appropriate for the Sageev construction. 
	Keeping this in mind, we shall, however, 
	implement a more explicit construction. 
	It turns out the cubical geometry of the cube complex 
	will be so well understood
	that one can nicely recognize its virtual specialness.
	
	In Section \ref{Sec-Prelim}, we provide background materials related to our discussion. In Section \ref{Sec-Eqns},
	we introduce the notion of multitwists and their configuration graphs. We setup
	the current equations and the BKN equations, and show the virtual equivalence of the existence of
	their nondegenerate symmetric solutions. In Section \ref{Sec-SpCubulation}, we construct a special cube complex homotopy
	equivalent to the mapping torus of any multitwist with a bipartite configuration graph.
	In Section \ref{Sec-mainProof}, we combine all the partial results above and provide a proof of Theorem \ref{main-npcGM}.
	
	\bigskip\noindent\textbf{Acknowledgement}. The author thanks Ian Agol and Piotr Przytycki
	for valuable communications, and Stefan Friedl for suggestion on the statement of Theorem
	\ref{main-npcGM}. The author also thanks the referees for proofreading and comments.

\section{Preliminaries}\label{Sec-Prelim}

	In this section, we recall some notions and known results 
	relevant to our discussion. As we are only interested in
	virtual properties in this paper, 
	we shall always work with orientable $3$-manifolds for convenience. 
	We refer to \cite{Ja} for standard terminology and facts in $3$-manifold topology.

	\subsection{Nonpositively curved $3$-manifolds}\label{Subsec-npc}
		A compact $3$-manifold is said to be \emph{nonpositively
		curved} if it supports a complete Riemannian metric in the interior with
		everywhere nonpositve sectional curvature. This is equivalent
		to being virtually nonpositively curved, (\cite{KL}). 
		It is well known that nonpositively curved $3$-manifolds are all
		aspherical. 
		
		Atoroidal aspherical $3$-manifolds are hyperbolic by the Geometrization,
		so these are clearly nonpositively curved.
		A Seifert fibered space is nonpositively curved if and only if
		it is virtually a product of the circle
		and a compact surface other than the sphere, (\cite{GW}). A torus
		bundle with an Anosov monodromy, and hence
		any $\Sol$-geometric $3$-manifold, is not nonpositively curved,
		(\cite{GW}, \cite{Ya}).
		By a result of Bernhard Leeb, any aspherical compact
		$3$-manifold with either nonempty boundary
		or at least one atoroidal JSJ piece is nonpositively curved, 
		(\cite{Le}). He also discovered the first family of nontrivial
		graph manifolds other than the $\Sol$-manifolds
		which are not nonpositively curved. In \cite{BS},
		Sergei Buyalo and Pavel Svetlov provided several characterizations
		of nonpositively curved graph manifolds in terms of certain structure
		equations and operators associated to the JSJ graph,
		(cf.~Subsection \ref{Subsec-BKNEqns}).

	\subsection{Graph manifolds}\label{Subsec-graphMfd}
		A \emph{graph manifold} is an irreducible compact orientable $3$-manifold 
		obtained by gluing up compact Seifert fibered spaces along boundary.
		We say a graph manifold is \emph{nontrivial} 
		if it is not a Seifert fibered space, or in other words,
		if it has a nontrivial JSJ decomposition.
		Recall that for any orientable irreducible compact $3$-manifold, the Jaco-Shalen-Johanson (JSJ) decomposition
		provides a minimal finite collection of essential tori, canonical up to
		isotopy, which cuts the manifold into atoroidal and Seifert fibered {\it pieces}.
		Graph manifolds are precisely compact $3$-manifolds with only Seifert-fibered pieces.
		Note that if a $3$-manifold is virtually
		a torus bundle with an Anosov monodromy, it is a nontrivial graph manifold
		according to our convention,
		but in literature (e.g.~\cite{LW}), some authors exclude this
		case, and use the term
		`nontrivial' in a the sense of having nontrivial JSJ decomposition
		with Seifert-fibered JSJ pieces over hyperbolic bases.
		
		Much has been explored on virtual properties of graph manifolds during the past two decades.
		Here we are most interested in virtual fibering and nonpositive curving. 
		In \cite{LW}, John Luecke and Yingqing Wu exhibited examples of
		graph manifolds that are not 
		virtually fibered over the circle.  There are also graph manifolds that are
		fibered but not nonpositively curved, for example, the mapping torus $M_\sigma$
		of an automorphism $\sigma$ of a closed oriented surface $F$ of nonpositive Euler
		characteristic, where $\sigma$ is the product of
		positive powers of right-hand Dehn twists along a nonempty collection of
		disjoint, mutually non-parallel simple closed curves on $F$, (\cite[Theorem 3.7]{KL}). 
		However, Pavel Svetlov proved:
		
		\begin{theorem}[{\cite{Sv}, cf.~also \cite{WY} for the bounded case}]\label{Thm-Sv} 
			Nonpositively curved graph manifolds virtually fiber over the circle.
		\end{theorem}
		
		Svetlov's proof relies on a criterion of Hyam Rubinstein and Shicheng Wang on 
		the separability of immersed horizontal surfaces in graph manifolds.
		In \cite{BS}, this approach was systematically developed, 
		providing characterizations of various virtual properties of graph manifolds,
		on three levels in terms of homological, geometric, and operator-spectral data, respectively. For
		our purpose the second
		level is particularly interesting, as it
		describes the nonpositively curved metric on the graph manifold,
		via the so-called BKN equations. We shall investigate this system of equations in more details
		in Section \ref{Sec-Eqns}. The paper \cite{BS} is written in an elegant, expository style, 
		so we recommend it for further information of the reader.

	\subsection{Special cube complexes}\label{Subsec-SpCubeCplx}
		Special cube complexes were introduced by
		Fr\'{e}d\'{e}ric Haglund and Daniel Wise (cf.~\cite{HW}) as suitable generalizations of graphs in higher dimensions.
		Heuristically speaking, these are cube complexes in which hyperplanes are `properly embedded in
		properly transversal positions' in certain combinatorial sense.
		
		The building blocks in a cube complex are Euclidean \emph{$n$-cubes}, namely,
		isometric copies of the double-unit cube $[-1,1]^n\subset\RR^n$ for
		each dimension $n\geq0$. For all dimensions $n>0$,
		every $n$-cube contains exactly $n$ \emph{midcubes} which are
		the $(n-1)$-cubes obtained from $[-1,1]^n$ intersecting with each
		coordinate hyperplane through the origin of $\RR^n$, and $2^n$ 
		\emph{corners}, which are the $n$-simplices each spanned by the half-edges
		adjacent to a single vertex.
		
		\begin{definition}\label{cubeCplx} 
			A \emph{cube complex} $X$ is a CW complex whose cells are identified as cubes, 
			such that the attaching maps are combinatorial in the sense that they
			send the faces of cubes isometrically onto lower-dimensinoal cubes. 
			For any vertex $v\in X$, the \emph{link} $\lk(v)$ of $v$ is 
			the complex of simplices over which
			the union of the corners-of-cubes at $v$ 
			can be naturally regarded as a cone.
			
			A \emph{combinatorial map} $f:X\to Y$
			between cube complexes is a cellular map, whose restriction
			to each cube is an isometric homeomorphism
			onto the image.
			It is a (combinatorial) \emph{local isometry}, often written as $f:X\looparrowright Y$, if
			for any vertex $v\in X$, the induced combinatorial map between links 
			$\dot{f}_v:\,\lk(v)\to\lk(f(v))$ is furthermore 
			an embedding whose image is a full subcomplex,
			(i.e. such that a simplex lies in the image whenever all its vertices
			lie in the image). 
		\end{definition}
		
		\begin{remark} 
			A cube complex is sometimes said to be \emph{simple} if all the links
			are simplicial complexes. For square (i.e.~$2$-dimensional cube) complexes,
			this is equivalent to saying that no link contains a cycle of at most two edges.
			Nonpositively curved cube complexes are all simple. A
			\emph{nonpositively curved} cube complex is known as a cube complex whose
			universal covering is $\CAT(0)$ with respect to the naturally induced path
			metric. An equivalent combinatorial characterization is that 
			all the links are flag (i.e.
			such that any finite collection of pairwisely joined 
			vertices spans a unique simplex), (\cite{Gr}). 
		\end{remark}
		
		The collection of all the midcubes in a cube complex $X$ forms a new cube complex whose connected components
		are called \emph{hyperplanes}. A hyperplane $H$ is \emph{two-sided} if it naturally induces
		a local isometry $H\times[-1,1]\looparrowright X$.  
		A hyperplane $H$ is said to be \emph{self-intersecting} if there are two edges dual to (i.e.~across)
		$H$ lying in a common square of $X$; or it is said to be (directly) \emph{self-osculating} if there are
		two dual edges not in a common square but adjacent at a common vertex, whose directions 
		towards the vertex induces the same side of $H$ in an obvious sense. Two distinct
		hyperplanes $H$, $H'$ are said to be \emph{inter-osculating} if they both intersect and osculate, namely, such that
		there are two edges dual to $H$ and $H'$ respectively, which share a common square of $X$, and that there
		are two edges dual to $H$ and $H'$ respectively, which do not share a common square in $X$ but share a common vertex.
		
		\begin{definition}\label{spCubeCplx} 
			A cube complex $X$ is said to be \emph{special} 
			if its hyperplanes are all two-sided, with no self-intersection, no self-osculation,
			nor inter-osculation.
		\end{definition}
		
		\begin{remark}\label{A-special} 
			In earlier literatures, (e.g.~\cite{HW}), 
			this property was sometimes refered to as being `A-special', indicating
			its relation to right-angled \underline{A}rtin groups.
		\end{remark}
		
		Prototype examples of special cube complexes are graphs and the cube complexes
		of right-angled Artin groups, (cf.~Subsection \ref{Subsec-RAAG}). It is also known that every
		special cube complex can be canonically completed by `filling up cubes' 
		into a nonpositively curved special cube complex.

	\subsection{Right-angled Artin groups}\label{Subsec-RAAG}
		Let $\Gamma$ be a (possibly infinite, and possibly locally infinite)
		simplicial graph. Denote the vertex set as $\vrtx(\Gamma)$, and the
		edge set as $\edge(\Gamma)$. The \emph{right-angled Artin group} associated
		to $\Gamma$ is the group:
			$$A(\Gamma)\,=\,\langle x_v\,:\,v\in\vrtx(\Gamma)\,|\, [x_v,x_{v'}]\,:\,\{v,v'\}\in\edge(\Gamma)\rangle.$$
		Free groups and abelian groups are examples of
		right-angled Artin groups, associated to edgeless graphs and
		perfect graphs, respectively.
		
		Right-angled Artin groups are closely related to special cube complexes. For any 
		right-angled Artin group $G$, there is a special 
		nonpositively curved complex $\mathcal{C}(\Gamma)$ by completing the 
		square complex naturally associated to the defining 
		presentation of $A(\Gamma)$. It has a single vertex,
		and the edges of $\mathcal{C}(\Gamma)$ correspond to the generators
		$x_v$'s, and the squares of $\mathcal{C}(\Gamma)$ correspond to the relators $[x_v,x_{v'}]$'s. Moreover,
		Haglund and Wise proved:
		
		\begin{theorem}[{\cite[Theorem 4.2]{HW}, cf.~Remark \ref{A-special}}]\label{Thm-HW} 
			A cube complex is special
			if and only if it admits a local isometry into the
			cube complex of a right-angled Artin group. 
		\end{theorem}
		
		In fact, the right-angled Artin group associate to a special cube complex $X$ can be canonically
		constructed by taking the simplicial graph whose vertices are hyperplanes in $X$, such that
		two vertices are joined by an edge if and only if the corresponding hyperplanes intersect
		each other. Note Theorem \ref{Thm-HW} does not assume the cube complex to be finite dimensional,
		nor the right-angled Artin group to be finitely generated.
		
		Finitely generated right-angled Artin groups are known to be subgroups of finitely generated right-angled Coxeter groups, 
		which are subgroups of ${\rm SL}(n,\ZZ)$ for some integer $n>0$, (cf.~\cite{DJ}, \cite{HsW}).

	\subsection{Agol's RFRS condition}\label{Subsec-RFRS}
		In \cite{Ag}, Ian Agol gave a criterion for $3$-manifolds to fiber over the circle in terms
		of a residual-type condition on the fundamental groups.
		
		\begin{definition}\label{RFRS} 
			A group $G$ is said to be \emph{residually finite
			rationally solvable} (or \emph{RFRS}) if there is a sequence of finite index subgroups
			$G=G_0\triangleright G_1\triangleright\cdots$ such that $\cap_{i\geq 0} G_i=\{1\}$, and that
			the quotient homomorphism $G_i\to G_i\,/\,G_{i+1}$ factors through a free abelian group
			for every integer $i\geq0$.
		\end{definition}
	
		Examples of virtually RFRS groups are surface groups, reflection groups, 
		and finitely generated right-angled Artin groups, following from the more general fact
		that finitely generated right-angled Coxeter groups are virtually RFRS (\cite[Theorem 2.2]{Ag}). 
		
		\begin{theorem}[{Cf.~\cite[Theorem 5.1]{Ag}}]\label{Thm-Ag}
			Let $M$ be an orientable compact irreducible $3$-manifold of zero Euler characteristic. If
			$\pi_1(M)$ is virtually RFRS, then $M$ virtually fibers over the circle. 
		\end{theorem}
			
		Roughly speaking, 
		the RFRS assumption allows one to pass to a suitable sequence of finite cyclic coverings
		by cutting and pasting the $3$-manifold along nonseparating subsurfaces. This
		process can be shown to destroy all the nonproduct parts within a finite
		number of steps,
		using Gabai's sutured
		manifold hierarchy. With his criterion, Agol showed that
		reflection $3$-orbifolds and arithmetic hyperbolic $3$-manifolds
		defined by quadratic forms virtually fiber. His criterion is also an ingredient
		in Wise's proof of the Virtual Fibering Conjecture 
		for closed Haken hyperbolic $3$-manifolds.

\section{Current equations versus BKN equations}\label{Sec-Eqns}
	In this section, we introduce the system of the current equations and the system
	of the BKN equations associated to
	multitwists. A nondegenerate symmetric solution of the current equations reflects the virtual
	RFRS-ness of the mapping torus, (Lemma \ref{survivalCriterion});
	and a nondegenerate symmetric solution of the BKN equations reflects the
	nonpositively-curving property, (Theorem \ref{Thm-BS}). We show that
	the existence of such a solution for any one of these two systems virtually implies that
	for the other, (Proposition \ref{nsSolutions}).

	\subsection{Mapping tori of multitwists}\label{Subsec-multitwist}
		Let $F$ be an oriented closed surface of negative Euler characteristic. For any 
		essential simple closed curve $z$ on $F$, we denote the (right-hand) Dehn twist 
		on $F$ along $z$ as:
			$$D_z:F\to F.$$ 
		Note it does not depend on the direction
		of $z$. It induces 
		the automorphism of $H_1(F)$ defined by
		$D_{z*}(\alpha)\,=\,\alpha\,+\,I([z],\alpha)\,[z]$
		for any $\alpha\in H_1(F)$, where $I:H_1(F)\times H_1(F)\to\ZZ$ denotes the intersection form.
		
		\begin{definition}\label{multitwist} 
			Let $F$ be an oriented closed surface of negative Euler characteristic,
			and $z_1,\cdots,z_s$ be a collection of mutually disjoint, mutually non-parallel essential 
			simple closed curves, and $b_1,\cdots, b_s$ be nonzero integers. The \emph{multitwist} $\sigma$
			on $F$ along these curves with multiplicity
			$b_i$'s is the automorphism (i.e.~orientation-preserving
			self-homeomorphism) defined by $\sigma=D_{z_1}^{b_1}\cdots D_{z_s}^{b_s}$.
			We often more specifically denote a multitwist by the pair $(F,\sigma)$.
		\end{definition}
		
		The JSJ decomposition of the mapping torus $M_\sigma$ of a multitwist $(F,\sigma)$ is fairly easy to describe. 
		Recall that for an oriented closed surface $F$ with an
		automorphism $\sigma$, the mapping torus $M_\sigma$ associated
		to the pair $(F,\sigma)$ is the closed oriented $3$-manifold obtained from $F\times[0,1]$ identifying
		$(x,0)$ with $(\sigma(x),1)$ for every $x\in F$. 
		With the notations above, a JSJ torus of $M_\sigma$ is the suspension of some $z_i$, namely, the image
		of $z_i\times[0,1]$ in $M_\sigma$; and a JSJ piece is the suspension of some component of $F-z_1\cup\cdots\cup z_s$ (i.e.~the compact surface
		obtained from $F$ by removing open regular neighborhoods
		of $z_i$'s). 
		We shall usually denote the union of JSJ tori of $M_\sigma$ as $\tori$, and the disjoint union
		of (compact) JSJ pieces as $\pieces$. 
		There is a naturally induced graph (i.e.~a CW $1$-complex):
			$$\cfggraph=\cfggraph(\pieces,\tori),$$ 
		dual to the decomposition,
		called the \emph{JSJ graph}, whose vertex 
		set $\vrtx(\cfggraph)$ labels to the JSJ pieces, and whose edge set $\edge(\cfggraph)$ labels to the JSJ tori.
		For any JSJ piece $J_v\subset\pieces$ corresponding to a vertex $v\in\vrtx(\cfggraph)$,
		each component of $\partial J$ is a torus $T_\delta$'s labelled by an end-of-edge $\delta$ adjacent to $v$.
		We often denote the set of ends-of-edges as $\widetilde\edge(\cfggraph)$, which may be regarded as a two-fold
		covering of $\edge(\cfggraph)$. For any end-of-edge $\delta$, there is a natural orientation-reversing homeomorphism:
			$$\phi_\delta:T_\delta\to T_{\bar\delta},$$
		where $\bar\delta$ is the opposite end of the edge which $\delta$ belongs to, satisfying
		$\phi_{\bar\delta}=\phi_\delta^{-1}$. These $\phi_\delta$'s together defines an orientation-reversing involution:
			$$\phi:\partial\pieces\to\partial\pieces,$$
		often referred to as the \emph{gluing}. In view of this, we also
		re-index the curves $z_i$'s and the multiplicity $b_i$'s by $\edge(\cfggraph)$. In the rest of this paper,
		it will be convenient to
		write $b_\delta$ for $b_e$ whenever $\delta$ is an end of $e$. Thus:
			$$b_\delta=b_{\bar\delta}.$$
		We also index the components of $F-\cup_{e\in\edge(\cfggraph)}\,z_e$ by $\vrtx(\cfggraph)$, writing
		$F_v$ for the component indexed by $v\in\vrtx(\cfggraph)$.
		
		We may integrate the defining information for a multitwist alternatively as follows.
		By a \emph{cycle} $c$ of $\cfggraph$, we mean a sequence of mutually distinct, consecutive
		directed edges $\vec{e}_1,\cdots,\vec{e}_r$ such that the terminal endpoint of $\vec{e}_r$
		coincides with the initial endpoint of $\vec{e}_1$. 
		A graph $\cfggraph$ is said to be \emph{bipartite} if every cycle of $\cfggraph$ consists of an even number
		of edges.
		
		\begin{definition} 
			Let $(F,\sigma)$ be a multitwist. 
			The \emph{configuration graph} associated to $(F,\sigma)$ is the triple:
				$$(\cfggraph,\{\chi_v\},\{b_e\}),$$
			where $\cfggraph$ is the JSJ graph of $M_\sigma$, and the integer
			$\chi_v<0$ is the Euler characteristic of $F_v$
			for each vertex $v\in\vrtx(\cfggraph)$, and the integer $b_e\neq 0$ is the Dehn-twist multiplicity
			along the curve $z_e$ for each edge $e\in\edge(\cfggraph)$. A \emph{bipartite} configuration
			graph is such that $\cfggraph$ is bipartite.
		\end{definition}
		
		\begin{remark} 
			It is clear that the positive integer
			$-\chi_v$ is greater than the valence of $v$ minus one. Given a 
			graph decorated with nonzero integers satisfying this compatibility, 
			one may canonically recover a multitwist
			$(F,\sigma)$. However, it turns out that $\chi_v$'s are irrelevant to 
			the virtual properties of $M_\phi$ listed in Theorem \ref{main-npcGM}.
		\end{remark}
		
		Every JSJ piece $J_v$ of $M_\sigma$ has a natural product structure
		$F_v\times S^1$, where $F_v\times *$ are parallel to
		the copy $F_v\times\{0\}$, and where $*\times S^1$ are the Seifert fibers directed forward along the suspension flow. 
		Thus every (preglue) JSJ torus $T_\delta\subset\partial J_v$ has a canonical \emph{Waldhausen basis}, namely: 
			$$(f_\delta,\,z_\delta),$$ 
		where $f_\delta$ is the directed Seifert fiber, and $z_\delta\subset\partial F_v$ 
		is component with the induced orientation. In other words, $[f_\delta]$ and $[z_\delta]$ forms a basis of $H_1(T_\delta)$
		with the intersection number $I([f_\delta],\,[z_\delta])=1$.
		By a straightforward computation, $\phi_\delta*:H_1(T_\delta)\to H_1(T_{\bar\delta})$ is given by the formula:
			$$\begin{array}{ccccc}
				\phi_{\delta*}([f_\delta])	&=	&	[f_{\bar\delta}]&+	&b_\delta\,[z_{\bar\delta}],\\
				\phi_{\delta*}([z_\delta])	&=	&                   &	&-\,[z_{\bar\delta}].
			\end{array}$$
		In particular,  $b_\delta$ equals the intersection number between Seifert fibers from the two sides of $T_\delta$, namely,
		$b_\delta\,=\,I([f_\delta],\,\phi_{\bar\delta*}[f_{\bar\delta}])$ on $H_1(T_\delta)$. In fact, the latter is the 
		usual definition of $b_\delta$ (or $b_e$) for a graph manifold in general, (cf.~\cite[Subsection 1.3]{BS}).
		
		For every vertex $v\in\vrtx(\cfggraph)$, the \emph{charge} at $v$ is known as the rational number:
			$$k_v=\sum_{\delta\in\widetilde\edge(v)}\,\frac{1}{b_\delta}.$$
		It is the Euler number of the framing on $\partial J_v$ provided by the Seifert fibers from the adjacent pieces, with respect
		to its own Seifert fibration, or more precisely, the unique $k_v$ satisfying 
		$\sum_{\delta\in\widetilde\edge(v)}\,\frac{\phi_{\bar\delta*}[f_{\bar\delta}]}{b_\delta}
		=k_v\cdot [f_v]$ in $H_1(J_v;\QQ)$, with our notations. The charge at JSJ pieces can be defined for general graph manifolds, but unlike
		for mapping tori of multitwists, it is not always determined by the intersection numbers $b_e$'s between fibers from both sides on the JSJ tori,
		(cf.~\cite[Subsection 1.3]{BS}).
		
		We end up this subsection by pointing out a simplification of
		our situation. We say a multitwist (or more generally,
		a surface aumorphism) $(\tilde{F},\tilde\sigma)$ \emph{covers}
		$(F,\sigma)$ if there is a covering map $\kappa:\tilde{F}\to F$ such that the diagram below commutes:
			$$\begin{CD}
				\tilde{F}@>\tilde{\sigma}>>\tilde{F}\\
				@V\kappa VV @V\kappa VV\\
				F@>\sigma>>F.
			\end{CD}$$
		Hence the mapping torus $M_{\tilde{\sigma}}$ covers $M_\sigma$ if 
		$(\tilde{F},\tilde\sigma)$ covers $(F,\sigma)$.
		
		\begin{lemma}\label{mappingTorus} 
			Every closed, virtually fibered, nontrivial graph manifold is 
			finitely covered by either the mapping torus $M_\sigma$ of an Anosov automorphism $(T^2,\sigma)$
			of the torus, or the mapping torus $M_\sigma$ of
			a multitwist $(F,\sigma)$ with a bipartite configuration graph.
		\end{lemma}
		
		\begin{proof} 
			If the closed graph manifold is virtually a torus bundle over the circle,
			it is a nontrivial graph manifold if and only if the monodromy is Anosov.
			If the virtual fiber has negative Euler characteristic,
			by the Nielsen-Thurston classification of surface automorphisms (\cite{Th-NT}),
			a closed, virtually fibered, nontrivial graph manifold is finitely covered by the mapping torus $M_\sigma$ of a 
			nontrivial multitwist $\sigma=D^{b_1}_{z_1}\cdots D^{b_s}_{z_s}$ 
			of an oriented closed surface $F$ of negative Euler characteristic. To ensure
			the bipartition, let $H_1(\cfggraph)\to\ZZ_2$ be the homomorphism
			defined by the parity counting of edges in cycles of $\cfggraph$. The kernel of $\pi_1(\cfggraph)\to H_1(\cfggraph)\to\ZZ_2$
			corresponds to a $2$-fold covering $\tilde\cfggraph$ of $\cfggraph$, which furthermore induces a $2$-fold covering
			$\kappa:\tilde{F}\to F$. Note $\sigma'=\sigma^2$ 
			acts trivially on $H_1(F;\ZZ_2)$ as any $D^2_{z_i}$ does. 
			We may lift each $D^{2b_i}_{z_i}$
			to a multitwist on $\tilde{F}$ along preimages of $z_i$, and hence obtain
			a lift $\tilde\sigma'$ of $\sigma'$ such that $\tilde\sigma'\circ\kappa=\kappa\circ\sigma'$.
			Note that $\tilde\cfggraph$ is bipartite, carrying the configurations of $(\tilde{F},\tilde\sigma')$,
			the pair $(\tilde{F},\tilde\sigma')$ provides a mapping torus as desired.
		\end{proof}

	\subsection{The current equations}\label{Subsec-currentEqns} 
		Let $(F,\sigma)$ be a multitwist with
		a bipartite configuration graph. In this subsection, we introduce a system of
		\emph{current equations} associated to $(F,\sigma)$,
		whose nondegenerate symmetric solutions give rise to 
		nonseparating subsurfaces of $M_\sigma$ intersecting each $z_e$ 
		algebraically nontrivially.
		
		Remember a cycle $c$ of $\cfggraph$ is a combinatorial closed path
		formed by a finite number of consecutive directed edges $\vec{e}_1,\cdots,\vec{e}_r$.
		For any cycle $c$, the set of the initial ends of all the $\vec{e}_i$'s will be denoted
		as $\widetilde\edge(c)$.
		
		\begin{definition}\label{currentEqns} 
			Let $(F,\sigma)$ be a multitwist 
			with a bipartite configuration graph $(\cfggraph,\{\chi_v\},b_v)$.
			The system of the \emph{current equations} associated to $(F,\sigma)$, with the unknowns $x_\delta\in\RR$ 
			where $\delta\in\widetilde\edge(\cfggraph)$,
			consists of the following:
			\begin{itemize}
				\item For any vertex $v$ of $\cfggraph$, 
					$$\sum_{\delta\in\widetilde\edge(v)}\,x_\delta\,=\,0;$$
				\item For any cycle $c$ of $\cfggraph$,
					$$\sum_{\delta\in\widetilde\edge(c)}\,b_\delta x_\delta\,=\,0.$$
			\end{itemize}
			A solution $\{x_\delta\}$ of the current equations is said to be \emph{symmetric} if $x_\delta=-x_{\bar\delta}$ for
			each end-of-edge $\delta$, and to be \emph{nondegenerate} if no $x_\delta$ equals zero.
		\end{definition}
		
		\begin{remark} 
			As the equations are linear with integral coefficients, 
			there is a nondegenerate symmetric solution over $\RR$ if and only if
			there is such a solution over $\QQ$. Somewhat coincidentally,
			this system of equations is analogous to the Kirchhoff's circuit laws in physics if one
			allows `negative resistence', which explains its name.
		\end{remark}
		
		\begin{lemma}\label{survivalCriterion} 
			Let $(F,\sigma)$ be a multitwist 
			with a bipartite configuration graph. Then
			the current equations associated to $(F,\sigma)$ have a nondegenerate symmetric solution
			if and only if for each edge $e\in\edge(\cfggraph)$, the simple closed curve $z_e\subset F$
			is nontrivial in $H_1(M_\sigma;\QQ)$ under the natural inclusion of $F$ 
			into the mapping torus $M_\sigma$.
		\end{lemma}
		
		\begin{proof} 
			Because $\cfggraph$ is bipartite, we may pick a bicoloring $\varepsilon:\vrtx(\cfggraph)\to\{\pm1\}$,
			so that vertices of the same color have no edge between each other. This induces directions of the edges, say
			toward positive vertices, so $z_e$'s are also oriented.
			As every vertex is adjacent to either only incoming edges or only outgoing edges,
			a symmetric solution of the current equations at vertices exists if and only if there is
			a homology class $[s]\in H_1(F;\QQ)$ with the intersection number
			$I([s],\,[z_e])\,=\,x_{e}$, where
			$x_{e}$ is the $x_\delta$ for the initial end $\delta$ of $e$. 
			When the current equations around cycles are also fulfilled, 
			$\sigma_*[s]=[s]+\sum_{e\in\edge(\cfggraph)}\,b_ex_e\,[z_e]$ on $H_1(F;\QQ)$.
			This equals $[s]$ because the second term vanishes as it evalutes zero on every cycle.
			Thus we may suspend $[s]\in H_1(F;\QQ)$ to be a class $[S]\in H_2(M_\sigma;\QQ)$, so that
			the intersection pairing on $H_*(M_\sigma;\QQ)$ gives the intersection
			number $I([S],\,[z_e])\,=\,x_e$.
			Therefore, the existence of a symmetric solution is equivalent to the existence of
			$[S]$ which yields $x_e$'s under pairing with $[z_e]$'s, and the solution is
			nondegenerate if and only if $[z_e]$'s are nontrivial in $H_1(M_\sigma;\QQ)$.
		\end{proof}
		
		An immediate consequence is the following:
		
		\begin{lemma}\label{RFRSnsSol}
			Let $(F,\sigma)$ be a multitwist 
			with a bipartite configuration graph.
			If $M_\sigma$ is virtually RFRS, then for some positive integer $m$,
			and for some multitwist
			$(\tilde{F},\tilde{\sigma}^m)$ finitely covering the $m$-th power $(F,\sigma^m)$, 
			the current equations associated
			to $(\tilde{F},\tilde{\sigma}^m)$ admit a nondegenerate symmetric solution.
		\end{lemma}
		
		\begin{proof} 
			By passing to a finite covering, we may assume that the fundamental group
			of $M=M_\sigma$ is itself RFRS.
			Let $\pi_1(M)=G_0\triangleright G_1\triangleright\cdots$ be a cofinal filtration
			of finite index subgroups as in Definition \ref{RFRS}, and let $\tilde{M}_i$ be the 
			finite regular covering of $M$ corresponding to $G_i$. For any curve $z_e$, 
			suppose, as an element
			of $\pi_1(M)$, $z_e$ lies in $G_i$ but not in $G_{i+1}$, 
			then any component of the preimage of $z_e$ in $\tilde{M}_i$ is rational
			homologically nontrivial. Hence this holds for any finite cover
			of $\tilde{M}_i$ as well.	
			As there are finitely many $z_e$'s, we may assume $\tilde{M}_i$ is a finite covering such
			that the components of the preimage of all $z_e$'s are rationally homologically nontrivial.
			Note that $\tilde{M}_i$ is the mapping torus of some automorphism $(\tilde{F},\tilde{\sigma})$
			covering $(F,\sigma)$ of finite degree, where $\tilde{F}$ is the regular covering
			of $F$ corresponding to the subgroup $\pi_1(F)\cap G_i$.
			By further passing to sufficiently high power $\tilde\sigma^m$ of $\tilde\sigma$, 
			we have that $(\tilde{F},\tilde{\sigma}^m)$ is a multitwist. It has
			a bipartite configuration	graph since this holds for $(F,\sigma)$, 
			so from above we conclude that it has a nondegenerate symmetric
			solution by Lemma \ref{survivalCriterion}.		
		\end{proof}
		
		In Figure \ref{figCurrent}, we have
		sketched a nondegenerate current on an orientable closed surface with a multitwist.
		The simple closed curves marked with integers indicate the $z_e$'s and $b_e$'s,
		while the directed green curves represents a nondegenerate \emph{current}, by
		which we mean a first homology class preserved by the multitwist that
		intersects each $z_e$ algebraically nontrivially.
		
		\begin{figure}[htb]
			\centering
			\includegraphics[scale=.8]{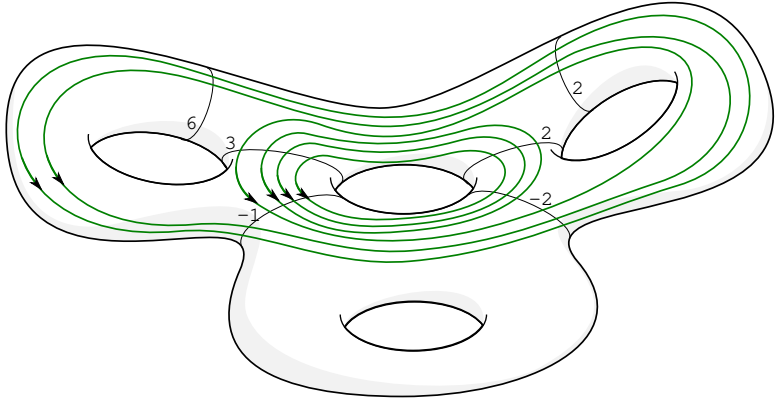}
			\caption{A nondegenerate current.}\label{figCurrent}
		\end{figure}

	\subsection{The BKN equations}\label{Subsec-BKNEqns} 
		For a nontrivial graph manifold, 
		the Buyalo-Kobel'ski\u{\i}-Neumann (BKN) equations are structure equations related to nonpositively curved metrics.
		Focusing on mapping tori of multitwists, we prefer rearranging them
		into the following simplified form:
		
		\begin{definition}\label{BKNEqns} 
			Let $(F,\sigma)$ be a multitwist 
			with a bipartite configuration graph $(\cfggraph,\{\chi_v\},\{b_e\})$.
			The system of (affine) the \emph{BKN equations} 
			associated to $(F,\sigma)$, with the unknowns $u_\delta\in(0,+\infty)$, and
			$\gamma_\delta\in [-1,1]$, where $\delta\in\widetilde\edge(\cfggraph)$,
			consists of the following:
			\begin{itemize}
				\item For any vertex $v$ of $\cfggraph$, 
					$$\sum_{\delta\in\widetilde\edge(v)}\,\frac{1-u_\delta\gamma_\delta}{b_\delta}\,=\,0;$$
				\item For any cycle $c$ of $\cfggraph$,
					$$\sum_{\delta\in\widetilde\edge(c)}\,\log(u_\delta)\,=\,0.$$
			\end{itemize}
			A solution $\{(u_\delta,\gamma_\delta)\}$ of the BKN equations is said to be \emph{symmetric} if 
			$u_\delta=\frac{1}{u_{\bar\delta}}$, and $\gamma_\delta=\gamma_{\bar\delta}$ for
			each end-of-edge $\delta$, and to be \emph{nondegenerate} if no $\gamma_\delta$ equals $\pm1$.
		\end{definition}
		
		\begin{remark}\label{generalBKN} 
			In general, the BKN equations can be defined for any nontrivial
			graph manifold $M$ with base-orientable JSJ pieces. Let the positive integer $|b_e|$ be 
			the absolute value of the
			intersection number of fibers from adjacent pieces on $T_e$
			for each edge $e\in\edge(\cfggraph)$, and the rational number $k_v$ be the charge at each vertex $v\in\vrtx(\cfggraph)$,
			(cf.~Subsection \ref{Subsec-multitwist}). With the unknowns
			$a_v\in[0,+\infty)$ for each vertex $v\in\vrtx(\cfggraph)$, and $\gamma_e\in[-1,1]$ for each edge $e\in\edge(\cfggraph)$,
			the BKN equations are defined as:
				$$k_va_v\,=\,\sum_{\delta\in\widetilde\edge(v)} \frac{a_{v(\bar\delta)}\gamma_{e(\delta)}}{|b_{e(\delta)}|},$$
			for every vertex $v\in\vrtx(\cfggraph)$, where $v(\bar\delta)$ means the vertex adjacent to the end-of-edge $\bar\delta$,
			and $e(\delta)$ means the edge carrying $\delta$, (\cite[Section 3]{BS}).
			We affinize the unknowns by taking 
			$u_\delta\,=\,\frac{a_{v(\bar\delta)}}{a_{v(\delta)}}$, and
			$\gamma_\delta\,=\,\mathrm{sgn}(b_{e(\delta)})\cdot\gamma_{e(\delta)}$ for each end-of-edge $\delta\in\widetilde\edge(\cfggraph)$.
			Note that $M$ virtually fibers if and only if $a_v$'s can be picked to be positive, (\cite[Theorem 3.1]{BS}).
			Note also that $k_v=\sum_{\delta\in{\widetilde\edge(v)}}\,\frac{1}{b_\delta}$ is the formula of
			the charge for mapping tori of multitwists,
			(Subsection \ref{Subsec-multitwist}).
			It is clear our definition above is equivalent to the original BKN equations.
		\end{remark}
		
		Besides several other results, in \cite{BS}, Buyalo and Svetlov
		provided a criterion for a graph manifold to be nonpositively curved. In particular,
		their result implies:
		
		\begin{theorem}[Cf.~{\cite[Theorem 3.1]{BS}}]\label{Thm-BS} 
			Let $(F,\sigma)$ be a multitwist 
			with a bipartite configuration graph. Then the mapping
			torus $M_\sigma$ is nonpositively curved if and only if
			the BKN equations associated to $(F,\sigma)$ have a nondegenerate symmetric solution.
			Moreover, in this case, there is a nondegenerate symmetric solution over $\QQ$. 
		\end{theorem} 
		
		To interpretate the unknowns geometrically, every nondegenerate symmetric solution of the 
		affine BKN equations endows $M_\sigma$ with a nonpositively curved metric
		so that the JSJ tori are totally geodesic, Eulidean, with 
		$u_\delta$ being the ratio between the lengths
		of the fibers $f_{\bar\delta}$ and $f_\delta$, 
		and $\gamma_\delta$ being the cosine of the angle 
		between these fibers.

	\subsection{Nondegenerate symmetric solutions}\label{Subsec-nsSolutions}
		In this subsection, we study the connection between the current equations 
		and the BKN equations. Heuristically speaking, a multitwist defines a 
		nonpositively curved mapping torus if and only if
		it virtually admits a nondegenerate current:
		
		\begin{prop}\label{nsSolutions} 
			Let $(F,\sigma)$ be a multitwist 
			with a bipartite configuration graph.
			Then the BKN equations associated to $(F,\sigma)$
			have a nondegenerate symmetric solution if
			and only if for some finite covering $(\tilde{F},\tilde\sigma)$ of $(F,\sigma)$,
			the current equations associated to $(\tilde{F},\tilde\sigma)$ have
			a nondegenerate symmetric solution.
		\end{prop}
		
		We prove Proposition \ref{nsSolutions} in the rest of this subsection. As before,
		we write the bipartite configuration graph of $(F,\sigma)$ as
		$(\cfggraph,\{\chi_v\},\{b_e\})$, and the essential curves on $F$ defining
		$\sigma$ as $z_e$'s, and the mapping torus of $(F,\sigma)$ as $M_\sigma$.
		
		The `if' direction follows from perturbing nondegenerate symmetric solutions
		of the current equations: 
		
		\begin{lemma}\label{curr2BKN} 
			If the current equations associated to
			$(F,\sigma)$ have a nondegenerate symmetric solution, so do the BKN equations.
		\end{lemma}
		
		\begin{proof} 
			Write $k_v=\sum_{\delta\in\widetilde\edge(\cfggraph)}\frac1{b_\delta}$ for the charge
			at a vertex $v\in\vrtx(\cfggraph)$. If all the charges $k_v$'s are zero, there is clearly a nondegenerate
			symmetric solution $\{(u^*_\delta,\gamma^*_\delta)\}$ to the BKN equations, where $u^*_\delta=1$ and
			$\gamma^*_\delta=0$. In the following, we suppose that not all the charges are zero.
			
			Let $\{x^*_\delta\}$ be a nondegenerate symmetric solution of the current equations.
			With the new parameters $t_v$'s and $\omega_\delta$'s in $\RR$, we put:
				$$u_\delta=\exp(t_{v(\delta)}-t_{v(\bar\delta)}),\textrm{ and }\gamma_\delta=\cos(\sqrt{\omega_\delta+\omega_{\bar\delta}}),$$
			where $v(\delta)$ denotes the vertex that $\delta$ is adjacent to,
			and where $\cos(\sqrt{z})$ stands for the analytic function
			$\sum_{j=0}^\infty \frac{(-1)^j}{(2j)!}z^j$.
			Then the BKN equations around cycles are automatically satisfied.
			
			To ensure the BKN equations
			at vertices, consider the functions:
				$$W_v(\vec{t},\vec{\omega})\,=\,
					\sum_{\delta\in\widetilde\edge(v)}\,
					\frac{1-\exp(t_{v}-t_{v(\bar\delta)})\cos(\sqrt{\omega_\delta+\omega_{\bar\delta}})}{b_\delta},$$
			for each vertex $v\in\vrtx(\cfggraph)$,
			where $\vec{t}$ and $\vec{\omega}$ are the vectors
			with the coordinates $t_v$'s and $\omega_\delta$'s, respectively. 
			Formally, these together define a smooth map:
				$$W:\,\RR^{\vrtx(\cfggraph)}\times\RR^{\widetilde\edge(\cfggraph)}\to\RR^{\vrtx(\cfggraph)},$$ 
			sending the origin to the origin. Our goal is to use the existence of
			$\{x^*_\delta\}$ to show that for some $\vec{t}$ sufficiently close to the origin, 
			there are $\vec\omega$ with sufficiently small positive coordinates,
			fulfilling $W_v(\vec{t},\vec{\omega})=0$ for any vertex $v\in\vrtx(\cfggraph)$.
			
			First observe that $W$ is a submersion near the origin. Indeed:
				$$\ud W_v|_{O}\,=\,\sum_{\delta\in\widetilde\edge(v)}\,\frac{-2\ud t_{v}+2\ud t_{v(\bar\delta)}+\ud \omega_{\delta}+\ud
					 \omega_{\bar\delta}}{2b_\delta},$$
			so the tangent map:
				$$\ud W|_O:T_O(\RR^{\vrtx(\cfggraph)}\times\RR^{\widetilde\edge(\cfggraph)})\,\to\, T_{W(O)}\RR^{\vrtx(\cfggraph)},$$
			satisfies at the origin:
			\begin{eqnarray*}
				\ud W\left(\frac{\partial}{\partial t_v}\right)
					&=
					&-k_v\,\frac{\partial}{\partial y_{v}}\,+\,
						\sum_{\delta\in\widetilde\edge(v)}\,\frac1{b_\delta}\,\frac{\partial}{\partial y_{v(\bar\delta)}},\\
				\ud W\left(\frac{\partial}{\partial\omega_\delta}\right)
					&=
					&\frac1{2b_\delta}\,\frac{\partial}{\partial y_{v(\delta)}}\,+\,
						\frac1{2b_\delta}\,\frac{\partial}{\partial y_{v(\bar\delta)}}.
			\end{eqnarray*}
			Here the $y_v$'s stand for the natural coordinate parameters of the target space $\RR^{\vrtx(\cfggraph)}$. 
			To see $\ud W|_O$ is surjective, note that by our assumption, 
			there is some vertex $u\in\vrtx(\cfggraph)$ at which the charge $k_u$ nonzero. Since the image of $\frac\partial{\partial
			t_u}\,-\,2\sum_{\delta\in\widetilde\edge(u)}\,\frac\partial{\partial\omega_\delta}$
			under $\ud W|_O$ equals $-2k_u\frac\partial{\partial y_u}$, we see $\frac\partial{\partial y_u}\in\mathrm{Im}(\ud W|_O)$.
			Thus for any vertex $v$ adjacent to $u$, $\frac\partial{\partial y_v}$ also lies in $\mathrm{Im}(\ud W|_O)$ by the formula of $\ud
			W(\frac\partial{\partial \omega_\delta})$
			for $\delta\in\widetilde\edge(v)$. Further applying
			the formula of $\ud W(\frac\partial{\partial \omega_\delta})$'s for other 
			end-of-edges, it follows from the connectivity of the graph $\cfggraph$ that
			for all $v\in\vrtx(\cfggraph)$, $\frac\partial{\partial y_v}$ lies in $\mathrm{Im}(\ud W|_O)$. 
			This means $\ud W|_O$ is surjective.
			Now as $W$ is a submersion near the origin,
			the preimage of the origin of $\RR^{\vrtx(\cfggraph)}$ is a smooth
			submanifold $\mathcal{V}$ of $\RR^{\vrtx(\cfggraph)}\times
			\RR^{\widetilde\edge(\cfggraph)}$ through the origin, of the dimension
			$|\widetilde\edge(\cfggraph)|$.
			
			Because $\{x^*_\delta\}$ satisfies the current equations around cycles, it is clear that there exist
			$l^*_v\in\RR$, for all $v\in\vrtx(\Lambda)$, such that $b_\delta x^*_\delta$ equals
			$l^*_{v(\delta)}-l^*_{v_{(\bar\delta)}}$ for each end-of-edge $\delta\in\widetilde\edge(\cfggraph)$.
			Consider the curve $\alpha:(-\epsilon,\epsilon)\to\RR^{\vrtx(\cfggraph)}\times\RR^{\widetilde\edge(\cfggraph)}$,
			where $\alpha(s)=(\vec{t}(s),\,\vec{\omega}(s))$ is defined by:
				$$t_v(s)\,=\,l^*_v s,
					\textrm{ and } 
					\omega_\delta(s)\,=\,\frac12\arccos^2\left(\frac1{\cosh(b_\delta x^*_\delta s)}\right).$$
			Since $\{x^*_\delta\}$ is symmetric, i.e.~ $x^*_{\bar\delta}=-x^*_\delta$,
			we have:
			\begin{eqnarray*}
				W_v(\alpha(s))	
					&=& \sum_{\delta\in\widetilde\edge(v)}\,
					\frac{
						1 - \exp\left(b_\delta x^*_\delta s\right)\cdot\left(\cosh(b_\delta x^*_\delta s)\right)^{-1}
					}
					{b_\delta}
				\\
				&=&	-\,\sum_{\delta\in\widetilde\edge(v)}\,\frac{\tanh(b_\delta x^*_\delta s)}{b_\delta}.
			\end{eqnarray*}
			Using the current equation at vertices, it is clear that all the $W_v(\alpha(s))$ vanish up to the second order,
			namely, $W_v(\alpha(s))=o(s^2)$. Thus there exists a curve $\alpha^*:(-\epsilon,\epsilon)\to\mathcal{V}$
			which is jet-$2$ equivalent to $\alpha$. Hence
			$\alpha^*(s)=(\vec{t}^*(s),\,\vec\omega^*(s))$ has the Taylor expansion near zero:
				$$t^*_v(s)\,=\,l^*_v s+o(s^2),
					\textrm{ and }
					\omega^*_\delta(s)\,=\,\frac{(b_\delta x^*_\delta)^2 s^2}2+o(s^2).$$
			Because $\{x^*_\delta\}$ is nondegenerate, $\omega^*_\delta(s)$ is strictly positive for any sufficiently small nonzero $s$.
			Therefore, for any sufficiently small nonzero $s$, the corresponding 
			$u_\delta^*=\exp\left(t^*_{v(\delta)}(s)-t^*_{v(\bar\delta)}(s)\right)$ and
			$\gamma_\delta^*=\cos\left(\sqrt{\omega_\delta^*(s)+\omega^*_{\bar\delta}(s)}\right)$
			gives a nondegenerate symmetric solution of the BKN equations.
		\end{proof}
		
		To prove the `only-if' direction, we need some ingredient from \cite{Sv}:
		
		\begin{lemma}\label{virtualFiberSurface} 
			If the BKN equations associated to
			$(F,\sigma)$ has a nondegenerate symmetric solution over $\QQ$, then there
			are linearly independent elements $[\theta^+_\delta],\,[\theta^-_\delta]\in H_1(T_\delta;\QQ)$
			for each end-of-edge $\delta\in\widetilde\edge(\cfggraph)$, subject to
			$[\theta^\pm_\delta]=\pm\phi_{\bar\delta*}([\theta^\pm_{\bar\delta}])$, and 
			there is a connected immersed orientable closed surface $j:S\looparrowright M_\sigma$ satisfying the
			following. For the preimage $j^{-1}(T_\delta)$
			of each $T_\delta$, the homology class 
			$j_*[j^{-1}(T_\delta)]$ is a scalar multiple of $[\theta_\delta]=[\theta^+_\delta]+[\theta^-_\delta]$,
			and for each component $\hat\theta\subset j^{-1}(T_\delta)$, the homology class
			$j_*[\hat\theta]$ is a scalar multiple of either $[\theta^+_\delta]$ or $[\theta^-_\delta]$.
			Moreover, $S$ is a virtual fiber, namely, there is a surface bundle finitely covering $M_\sigma$
			with the fiber covering $S$ up to homotopy.
		\end{lemma}
		
		\begin{proof} 
			We outline the proof in our context following \cite{Sv},
			(cf.~Remark \ref{generalBKN} for the translation of the notations). Suppose $\{(u_\delta,\gamma_\delta)\}$
			is a nondegenerate symmetric solution over $\QQ$ of the BKN equations associated to $(F,\sigma)$.
			In fact, the elements $[\theta^\pm_\delta]$ can be
			explicitly picked as:
				$$[\theta^\pm_\delta]=\frac{1\pm\gamma_\delta}{2b_\delta}\,(\phi_{\bar\delta*}[f_{\bar\delta}]\pm u_\delta\,[f_\delta]).$$
			Note these are linearly independent elements as $u_\delta>0$ and $-1<\gamma_\delta<1$. Moreover,
			it is straightforward to check the following, where $[\theta_\delta]=[\theta_\delta^+]+[\theta_\delta^-]$,
			and $I(\cdot,\,\cdot)$ denotes the intersection pairing on $H_1(T_\delta;\,\QQ)$:
			\begin{itemize}
				\item For each end-of-edge $\delta\in\widetilde\edge(\cfggraph)$, using the nondegeneracy: 
						$$I([f_\delta],\,[\theta_\delta^\pm])\,>\,0;$$
				\item For each end-of-edge $\delta\in\widetilde\edge(\cfggraph)$: 
						$$I([f_\delta],\,[\theta_\delta])\,=\,1;$$
				\item For each vertex $v\in\vrtx(\cfggraph)$, using the BKN equation at $v$
					and the formula of the charge $k_v$, (cf.~Subsection \ref{Subsec-multitwist}): 
						$$\sum_{\delta\in\widetilde\edge(v)}\,I\left([\theta_\delta],\,
						\frac{\phi_{\bar\delta*}[f_{\bar\delta}]}{b_\delta}\right)\,=\,k_{v};$$
				\item And for each end-of-edge $\delta\in\widetilde\edge(\cfggraph)$:
					$$\left|\frac{I(\phi_{\bar\delta*}[f_{\bar\delta}],\,[\theta_\delta^\pm])}{I([f_\delta],\,[\theta_\delta^\pm])}\right|
					\,=\,u_\delta.$$
			\end{itemize}
			The first three properties imply the existence of a connected, horizontal (i.e.~transverse to the
			fiber in each Seifert fibered piece), immersed, orientable closed surface $j:S\looparrowright M_\sigma$.
			Furthermore, one may require that for any component $\hat\theta\subset j^{-1}(T_\delta)$, 
			the homology class
			$j_*[\hat\theta]$ is a scalar multiple of either $[\theta^+_\delta]$ or $[\theta^-_\delta]$,
			and $j_*[j^{-1}(T_\delta)]$ is a scalar multiple of $[\theta_\delta]$,
			(\cite[Lemma 3.1]{Ne}; cf.~\cite[Proposition 1.4]{Sv}). 
			The last property together with the BKN equations around cycles imply that $S$ is a
			virtually embedded horizontal surface, and hence is a virtual fiber, 
			using a cocycle criterion due to Rubinstein and Wang (\cite{RW}; cf.~\cite[Proposition 1.4]{Sv}).
		\end{proof}

		\begin{lemma}\label{BKN2curr} 
			If the BKN equations associated to
			$(F,\sigma)$ have a nondegenerate symmetric solution, so do the current equations
			associated to some finite covering $(\tilde{F},\tilde{\sigma})$ of $(F,\sigma)$.
		\end{lemma}
		
		\begin{proof} 
			By Theorem \ref{Thm-BS}, we may assume the
			nondegenerate symmetric solution to be over $\QQ$. 
			Let $j:S\looparrowright M_\sigma$ be the virtual fiber surface
			as ensured by Lemma \ref{virtualFiberSurface}. Note that for any essential simple closed
			curve $z_\delta\subset T_\delta$ corresponding to an end-of-edge $\delta\in\widetilde\edge(\cfggraph)$,
			at least one of the intersection numbers $I([z_\delta],\,[\theta^\pm_\delta])$ is nonzero.
			Suppose $\kappa:\tilde{M}_\sigma\to M_\sigma$ is a regular finite covering in which a finite covering 
			$\tilde{S}$ of $S$ embeds as a fiber. Then every component $\tilde{z}$ of $\kappa^{-1}(z_e)$ intersects
			a component of the preimage $\kappa^{-1}(j(S))$, i.e.~a conjuate of 
			$\tilde{S}$ under the covering transformation, 
			algebraically nontrivially. In particular, $\tilde{z}$ is nontrivial 
			in $\tilde{M}_\sigma$. Note that any component $\tilde{F}$ of $\kappa^{-1}(F)$ is still a 
			fiber of $\tilde{M}_\sigma$. Then a further finite covering of $\tilde{M}_\sigma$ is a mapping torus
			associated to a multitwist $(\tilde{F},\tilde{\sigma})$. Now the defining essential
			curves of $\tilde{\sigma}$ cover the $z_e$'s, so they
			are nontrivial in $H_1(M_{\tilde\sigma};\QQ)$. 
			By Lemma \ref{survivalCriterion}, this is to say the current equations associated to
			$(\tilde{F},\tilde\sigma)$ have a nondegenerate symmetric solution.
		\end{proof}
		
		\begin{proof}[Proof of Proposition \ref{nsSolutions}]
			We apply Lemma \ref{curr2BKN} to $(\tilde{F},\tilde{\sigma})$
			in the `if' direction. Note the mapping torus
			$M_{\tilde\sigma}$ is nonpositively curved if and only
			if so is $M_\sigma$, (\cite{KL}), thus equivalently,
			the BKN equations associated to $(\tilde{F},\tilde\sigma)$ 
			have a nondegenerate symmetric solution if and only if so do the BKN equations
			associated to $(F,\sigma)$, (Theorem \ref{Thm-BS}). The `only-if' direction follows
			from Lemma \ref{BKN2curr}.
		\end{proof}

\section{Virtual special cubulation}\label{Sec-SpCubulation}

	In this section, we construct a virtual special cube complex homotopy equivalent
	to the mapping torus of a multitwist with a bipartite configuration graph,
	under the assumption that the associated current equation possesses a nondegenerate
	symmetric solution, (Proposition \ref{SpCubulation}). 
	As mentioned before, this assumption should be regarded as
	a seemingly weaker translation of the RFRS condition.
	
	\begin{prop}\label{SpCubulation} 
		Let $(F,\sigma)$ be a multitwist 
		with a bipartite configuration graph.
		Suppose the current equations associated to $(F,\sigma)$
		have a nondegenerate symmetric solution. Then 
		for some integer $0\leq n\leq
		|\vrtx(\cfggraph)|-2$,
		the product $M_\sigma\times \RR^n$ of the mapping torus $M_\sigma$ and
		the $n$-space $\RR^n$ is 
		homeomorphic to a virtually special cube complex.
	\end{prop}
	
	\begin{remark} 
		In fact, the special cube complex by our construction is complete
		with finitely many hyperplanes, but when $n>0$,
		it has no compact convex subcomplex homotopy equivalent to $\tilde{M}_\sigma$.
	\end{remark}
	
	We prove Proposition \ref{SpCubulation} in the rest of this section.
	Our cubulation is inspired by the Sageev construction.
	Generally speaking, from any finite collection of $\pi_1$-injective 
	surfaces, the Sageev construction allows 
	one to obtain a CAT(0) cube complex on which the fundamental group 
	of the $3$-manifold acts, (cf.~\cite{Ni}). This will be a proper action if one can find
	a `sufficient' collection, and the virtual specialness will be
	verified if one can show the separability and the double-coset
	separability of	the surface subgroups. This approach has been taken by
	Piotr Przytycki and Dani Wise recently, (\cite{PW-graph}). 
	Instead of that, in this paper,	
	we shall present an explicit construction, through which
	the virtual specialness can be seen directly.

	Throughout this section, 
	we suppose $(F,\sigma)$ satisfies the assumption of Proposition
	\ref{SpCubulation}. We shall write the configuration graph of
	as $(\cfggraph,\{\chi_v\},\{b_e\})$, and the essential curves on $F$ defining
	$\sigma$ as $z_e$'s. Regarding the mapping torus $M_\sigma$ as $F\times[0,1]$ 
	identifying $(x,0)$ with $(\sigma(x),1)$ for every $x\in F$.
	there is a natural inclusion:
		$$F\subset M_\sigma,$$
	identifying $F$ as the copy $F\times\{0\}$ in $M_\sigma$.
	The Poincar\'{e} dual of $[F]\in H_2(M_\sigma)$ will be written as:
		$$\alpha_F\in H^1(M_\sigma).$$
		
	\subsection{The cut-bind system $(\cuts,\binds)$}\label{Subsec-setupData} 
		In this subsection, we choose some setup data for
		the construction of the cubulation.
		It will be described as a cut-bind system $(\cuts,\binds)$
		of curves with respect to $(F,\sigma)$, which naturally induces
		a cohomology class $\bar\xi\in H^1(F)$ and a finite collection
		of lifted cohomology
		classes $\xi^1,\cdots,\xi^{n+2}\in H^1(M_\sigma)$ for
		some $0\leq n\leq |\vrtx(\cfggraph)|-2$.
		
		\begin{definition}\label{subordinate} 
			Suppose $\cuts\subset F$ is a collection of mutually disjoint 
			simple closed curves so that $\pants=F-\cuts$ is a disjoint union of
			(compact) pairs-of-pants.
			We say the pants decomposition $\pants=F-\cuts$ is 
			\emph{subordinate to} $(F,\sigma)$ if $z_e$ is a component of $\cuts$ for every
			$e\in\edge(\cfggraph)$. The components of $\cuts$ are called the \emph{cut-curves}.
		\end{definition}
		
		\begin{definition}\label{cbs} 
			Suppose $\pants=F-\cuts$ is a pants decomposition of $F$
			subordinate to $(F,\sigma)$, and $\binds\subset F$
			is a union of oriented mutually-disjoint simple closed curves, 
			representing a cohomology class $\bar\xi\in H^1(F)$ in Poincar\'{e} dual.
			If for every cut curve $z\in\cuts$, $\binds\cap z$ is nonempty consisting
			of exactly $|\bar\xi([z])|$ points, and if
			for every pair-of-pants $P\subset\pants$, $\binds\cap P$ is a union of disjoint arcs,
			then we call $(\cuts,\binds)$ a \emph{cut-bind system} with respect to
			$(F,\sigma)$. The components of $\binds$ are called the \emph{bind-curves}.
		\end{definition}
		
		\begin{lemma}\label{pantsDecomposition}
			With the notations above, there is a pants decomposition 
			$\pants=F-\cuts$ subordinate to 
			$(F,\sigma)$, so that every cut-curve $z\subset\cuts$ represents a nontrivial 
			element (up to sign) in $H_1(M_\sigma;\QQ)$. In particular, every cut-curve  $z
			\subset\cuts$ is non-separating on $F$, adjacent to two distinct pairs-of-pants of
			$\pants$.
		\end{lemma}
		
		\begin{proof} 
			Let $\cuts_0$ be the union of the $z_e$'s. 
			By Lemma \ref{survivalCriterion}, as the current equations associated to 
			$(F,\sigma)$ have a nondegenerate solution, each $z_e$ survives in $H_1(M_\sigma)$.
			Suppose by induction that for some $k\geq0$,
			$\cuts_k\subset F$ is a union of mutually non-parallel, disjoint essential simple closed curves
			that has been constructed, so that every component $z\subset\cuts$ survives in $H_1(M_\sigma;\QQ)$.
			
			Whenever $F-\cuts_k$ is not a disjoint union of pairs-of-pants, we construct $\cuts_{k+1}$ as follows.  
			If there is a non-planar component of $C$ of $F-\cuts_k$, let $z\subset C$ be a non-separating simple
			closed curve. It is clear that $z$ is rational-homologically nontrivial in $M_\sigma$, so
			we take $\cuts_{k+1}=\cuts_k\cup\{z\}$. Otherwise, there is still a component 
			planar $C$ of $F-\cuts_k$
			which has at least four boundary components. With the orientation induced from that of $C$, at least two
			boundary components $z',z''\subset\partial C$ such that $[z']+[z'']\neq0$ in $H_1(M_\sigma;\QQ)$.
			Let $z\subset C$ be a simple closed curve separating $z'\cup z''$ from the other components of $\partial C$,
			so $z$ is rational-homologically nontrivial in $M_\sigma$. We take $\cuts_{k+1}=\cuts_k\cup\{z\}$.
			
			This process terminates after finitely many steps. In the end, we get a pants decomposition
			$\pants=F-\cuts$ as desired.
		\end{proof}
		
		Fix a pants decomposition $\pants=F-\cuts$ subordinate
		to $(F,\sigma)$ as asserted by Lemma \ref{pantsDecomposition}.
		
		\begin{lemma}\label{barXi} 
			With the notations above, there 
			is a cohomology class $\bar\xi\in H^1(F)$ invariant under $\sigma_*$,
			such that $\bar\xi([z])$ is nontrivial for any cut-curve
			$z\subset\cuts$.
		\end{lemma}
		
		\begin{proof} 
			As the pants decomposition $\pants=F-\cuts$
			satisfies the conclusion of Lemma \ref{pantsDecomposition}, 
			we may find a generic $\xi^0\in H^1(M_\sigma;\QQ)$, so that $\xi^0([z])\neq0$ for any cutting
			curve $z\subset\cuts$.
			After passing to a multiple, we may assume $\xi^0\in H^1(M_\sigma)$ as well.
			Take $\bar\xi\in H^1(F)$ to be the restriction of $\xi^0$ under the natural
			inclusion $F\subset M_\sigma$. It is straightforward to check that $\bar\xi$ satisfies the
			conclusion.
		\end{proof}
		
		Fix a cohomology class $\bar\xi\in H^1(F)$ as asserted by Lemma \ref{barXi}.
		
		\begin{lemma}\label{cbsExist} 
			With the notations above,
			there is a cut-bind system $(\cuts,\binds)$ with respect to $(F,\sigma)$ such that
			$[\binds]\in H_1(F)$ is in Poincar\'{e} dual to $\bar\xi\in H^1(F)$.
		\end{lemma}
		
		\begin{proof} 
			Let $P\subset\pants$ be a pair-of-pants. For simplicity, 
			we write $v$ for the vertex $v(P)\in\vrtx(\cfggraph)$
			carrying a pair-of-pants $P\subset\pants$, 
			and write $i$ for the index $i(v)\in\{1,2,\cdots,n+2\}$.
			Note that the value of the cohomology class $\bar\xi$ on the three components $z,z',z''$
			of $\partial P$ (with the induced orientations) are three nonzero integers
			$m,m',m''$, respectively, so that $m+m'+m''=0$. Without loss of generality, we assume
			$m'$ and $m''$ has the same sign, and $m$ has the opposite sign. Thus
			one may take $|m'|$ parallel arcs
			from $z$ to $z'$, and $|m''|$ parallel arcs from $z$ to $z''$, so that they are mutually disjoint,
			cutting $P$ into one octagon and $|m|-2$ bands. See
			Figure \ref{figCubPants2D} for an illustration with 
			$m=5$, $m'=-2$, and $m''=-3$, in which the bind-arcs
			are colored green.
			
			\begin{figure}[htb]
				\centering
				\includegraphics[scale=1.2]{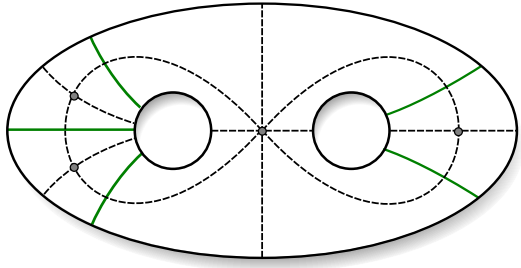}
				\caption{Bind-arcs in a pair-of-pants and the dual cubulation.}\label{figCubPants2D}
			\end{figure}
			
			There is a canonical orientation
			for these arcs so that their union is dual to the restriction $\bar\xi|_P\in H^1(P)$.
			These parallel arcs will later 
			serve as the bind-arcs $\binds\cap P$. Moreover, we glue up these $P$'s according to
			the pants decomposition, or more precisely, embed $P$ into 
			$F$, so that the bind-arcs match up along the boundary, giving
			rise to bind-curves $\binds\subset F$. Arbitrary matching up may
			result in some $[\binds]\in H_1(F)$ which differs from the dual of $\bar\xi\in H^1(F)$ by an
			integral linear combination of the $[z]$'s, where $z\subset\cuts$ are the cut-curves.
			However, we may modify the matching by shifting the embedding
			of $P$'s into $F$ near $\partial P$, so that $[\binds]$
			is the Poincar\'{e} dual of $\bar\xi$, as desired.
		\end{proof}
		
		Fix a cut-bind system $(\cuts,\binds)$ with respect to $(F,\sigma)$ as asserted by Lemma \ref{cbsExist}.
		
		\begin{lemma}\label{cohomologyClasses} 
			With the notations above, there are canonical integral cohomology classes:
				$$\xi^1,\cdots,\xi^{n+2}\in H^1(M_\sigma),$$
			for some $0\leq n\leq |\vrtx(\cfggraph)|-2$, such that  
			each $\xi^j$ restricts to $\bar\xi\in H^1(F)$ under the natural inclusion
			$F\subset M_\sigma$, and that for every vertex $v\in\vrtx(\cfggraph)$,
			there is exactly one $\xi^i$ with $\xi^i([f_v])=0$.
		\end{lemma}
		
		\begin{proof} 
			As $\bar\xi\in H^1(F)$ is invariant under $\sigma_*$, it comes from the
			restriction of some $\xi^0\in H^1(M_\sigma)$, and the space of such cohomology classes
			is exactly $\xi^0\,+\,\ZZ\,\alpha_F$. For any vertex $v\in\vrtx(\cfggraph)$, 
			and any $\xi(t)=\xi^0+t\,\alpha_F$, we have $\xi([f_v])=0$ if and only if 
			$t$ equals $-\xi^0([f_v])$, since $\alpha_F([f_v])=1$.
			We may take $\xi^1,\cdots,\xi^{n+2}$ to be all the $\xi(t)$'s, where
			$t=-\xi^0([f_v])$ as $v$ runs over $\vrtx(\cfggraph)$. This is the canonically
			(ordered) finite set of cohomology classes.
		\end{proof}
		
		Let $\xi^1,\cdots,\xi^{n+2}\in H^1(M_\sigma)$ be the cohomology
		as asserted by Lemma \ref{cohomologyClasses}. 
		It will be convenient to write: 
			$$i(v)\in\{1,2,\cdots,n+2\},$$
		for the unique index associated to $v\in\vrtx(\cfggraph)$.
		We often call $i(v)$ the \emph{vertical index} with respect to $v$,
		and call any other index $j$ a \emph{horizontal index} with respect to $v$.

	\subsection{The cube complex $X$}\label{Subsec-cubCplxX}
		In this subsection, we construct a canonical cube complex $X$
		provided a choice of the cut-bind system $(\cuts,\,\binds)$ 
		with respect to $(F,\sigma)$, as can be
		obtained by Subsection \ref{Subsec-setupData}.
		
		Let $\bar\xi\in H^1(F)$ and $\xi^1,\cdots,\xi^{n+1}\in H^1(M_\sigma)$
		be the canonically associated cohomology classes as before.
		We also have a natural cubulation of $F$ dual to $(\cuts,\binds)$,
		namely, the cubulation 
		obtained by matching up the dual cubulation on each pair-of-pants.
		See Figure \ref{figCubPants2D} for an illustration, 
		where the dots and the dashed arcs are 
		the vertices and the edges in the dual cubulation, respectively.
		Note that the boundary of the pair-of-pants is regarded
		as a union of mid-cubes.  
		
		\medskip\noindent\textbf{Step 1. Cubulate $J_v\times\RR^n$.} 
		For every vertex $v\in\vrtx(\cfggraph)$,
		the JSJ piece $J_v$ is the mapping torus of $\sigma$ restricted to $F_v\subset F$. 
		It has a canonical Seifert fibered structure given by the short exact sequence of groups:
			$$1\longrightarrow\pi_1(f_v)\longrightarrow\pi_1(J_v)\stackrel{p_v}\longrightarrow \pi_1(F_v)\longrightarrow 1,$$
		and a natural inclusion $F_v\subset J_v$. Hence $\pi_1(J_v)$ is a canonical direct product $\pi_1(F_v)\times\pi_1(f_v)$.
		
		Let $\widehat{F}_v$ be the universal covering of $F_v$ with the induced
		cubulation. 
		We parametrize $\RR^{n+2}$ by the real coordinates $\vec{\lambda}=(\lambda^1,\cdots,\lambda^{n+2})$, 
		and for each $1\leq i\leq n+2$, let $\RR^{n+1}_{(i)}$ be the subspace defined by $\lambda^i=0$. 
		We often denote points of $\RR^{n+1}_{(i)}$ by $\vec\lambda_{(i)}=(\lambda^1,\cdots,\widehat{\lambda^i},\cdots,
		\lambda^{n+2})$.
		The integral lattice of $\RR^{n+2}$ endows $\RR^{n+2}$, and each
		$\RR^{n+1}_{(i)}$, with a standard cubulation given by the tiling of standard unit $(n+2)$-cubes.
		Therefore, the fundamental group $\pi_1(J_v)$ acts on:
			$$\widehat{X}_v\,=\,\widehat{F}_v\times\RR^{n+1}_{(i(v))},$$ 
		via:
			$$g.(w,\vec\lambda_{(i(v))})\,=\,\left(p_v(g).w,\,\vec\lambda_{(i(v))}+\vec\xi_{(i(v))}([g])\right),$$
		for any $g\in\pi_1(J_v)$, and any $(w,\,\vec\lambda_{(i(v))})\in \widehat{X}_v$. Here $p_v(g)$
		acts by the deck transformations on $\widehat{F}_v$, and
		$\vec\xi_{(i(v))}([g])$ denotes $(\xi^1([g]),\,\cdots,\widehat{\xi^{i(v)}([g])},\cdots,\xi^{n+2}([g]))$, where
		$[g]\in H_1(J_v)$ is the corresponding homology class.
		
		\begin{lemma}\label{X_Pproduct}
			The action on $\pi_1(J_v)$ on the cube complex
			$\widehat{X}_v$ is proper and combinatorial.
			The quotient cube complex:
				$$X_v=\pi_1(J_v)\,\backslash\,\widehat{X}_v,$$
			is homeomorphic to $J_v\times\RR^n$.
			Moreover, there is a canonical isomorphism $\pi_1(X_v)\cong\pi_1(J_v)$
			induced by an $n$-plane bundle structure of $X_v$ over $J_v$.
		\end{lemma}
		
		\begin{proof} 
			The stabilizer of $\pi_1(J_v)$ at any point of $\widehat{F}_v$ is the center
			$\pi_1(f_v)\cong\ZZ$. Since $f_v$ acts on $\RR^{n+1}_{(i(v))}$ as  
			a translation by $\vec\xi_{(i(v))}([f_v])$, which is nontrivial by Lemma \ref{cohomologyClasses},
			the action of $\pi_1(J_v)$ is proper on $\widehat{F}\times\RR^{n+1}_{(i(v))}$.
			It clearly preserves the cube complex structure as well.
			
			Moreover, let $V_v$ be the subspace of $\RR^{n+1}_{(i(v))}$ spanned by $\vec\xi_{(i(v))}([f_v])$ over $\RR$.
			Then the $n$-dimensional foliation of $\RR^{n+1}_{(i(v))}$ by the affine parallel copies of 
			any subspace complementary $U$ to $V_v$ induces 
			an (affine) $n$-plane bundle structure of $X_v$ over $J_v$. In fact, different choices of
			$U$'s result in naturally isomorphic bundles over $J_v$. This bundle is trivial because the structure group
			is translational, and it induces a canonical isomorphism $\pi_1(X_v)\cong\pi_1(J_v)$.
		\end{proof}
		
		For any end-of-edge $\delta\in\widetilde\edge(\cfggraph)$ adjacent to $v$, we write:
			$$\partial_\delta X_v\subset \partial X_v,$$
		for the boundary component corresponding to $\delta$. It is a cube complex formed by mid-cubes,
		and $\pi_1(\partial_\delta X_v)\cong \pi_1(T_\delta)$ under the natural isomorphism $\pi_1(X_v)\cong
		\pi_1(J_v)$. To understand the cubulation on $\partial_\delta X_v$,
		remember we have cubulated $z_\delta$ as the quotient of $\RR^1$ by a translation by $\xi^{i(v)}([z_\delta])$. Thus, the universal covering
		of $\partial_\delta X_v$ can be naturally regarded as $\RR^{n+2}$ up to integral translations,
		where $\RR^{n+2}\,\cong\,\RR^1\times\RR^{n+1}_{(i(v))}$ is identified via plugging $\RR^1$ as the $i(v)$-th factor. 
		Using this identification, $\partial_\delta X_v$ is the quotient of
		$\RR^{n+2}$ under the action of $\pi_1(T_e)$ via:
			$$g.\vec\lambda\,=\,\vec\lambda+\vec\xi([g]),$$
		for any $g\in\pi_1(T_\delta)$, and any $\vec\lambda\in\RR^{n+2}$.
		Here $\vec\xi([g])$ denotes $(\xi^1([g]),\,\cdots,\xi^{n+2}([g]))$.
		In particular, any 
		$z_\delta$ translates by $\vec\Delta(\bar\xi([z_\delta]))
		= (\bar\xi([z_\delta]),\,\cdots,\,\bar\xi([z_\delta]))$ along the diagonal
		of $\RR^{n+2}$,
		and any $f_v$ translates only in the subspace $\RR^{n+1}_{(i(v))}$.
		
		Because integral translations on the $\RR^{n+1}_{(i(v))}$ factor of $\widehat{X}_v=
		\widehat{F}_v\times\RR^{n+1}_{(i(v))}$ commute with the action
		of $\pi_1(J_v)$ on $\widehat{X}_v$, this induces a group of
		automorphisms of the cube complex $X_v$. It is an abelian group
		naturally isomorphic to the quotient of $\ZZ^{n+1}$ modulo $\vec\xi_{(i(v))}([f_v])$.
		We shall call these automorphisms \emph{combinatorial translations} of $X_v$.
		
		\medskip\noindent\textbf{Step 2. Cubulate $M_{\sigma}\times\RR^n$.} 
		It is almost clear how to glue up these $X_v$'s. In the previous step,
		we identified $\partial_\delta X_v$ up to combinatorial
		translations as the quotient
		of $\RR^{n+2}$ by the action of $\pi_1(T_\delta)$,
		and under the identification,
		the action of $\pi_1(T_\delta)$ and
		$\pi_1(T_{\bar\delta})$ coincide on $\RR^{n+2}$.
		This implies the cubulations on the boundary components
		corresponding to any pair of opposite ends match up. Therefore, we may glue up
		all $X_v$'s along the boundaries with respect to $\cfggraph$, 
		matching up the cube complex structures, to obtain 
		a cube complex:
			$$X=\bigcup_{v\in\vrtx(\cfggraph)}\,X_v,$$
		which is clearly complete.
		
		However, the cube complex $X$ is determined only up to
		combinatorial translations along the boundary
		at this point. To further 
		make the matching canonical, we need to specify a reference point
		on each component of each $\partial X_v$. For this purpose, let: 
			$$\bar{h}:\cfggraph\to F,$$
		be a lift of the configuration graph
		$\cfggraph$ into $F$, namely, such that $\bar{r}\circ \bar{h}:\cfggraph\to\cfggraph$ is 
		homotopic to the identity where
		$\bar{r}:F\to\cfggraph$ is the natural collapsing. This gives 
		a base-point $\bar\nu_v=\bar{h}(v)$ on each $F_v$, and a base-point $\bar\nu_e=\bar{h}(e)\cap z_e$
		on each $z_e$, and a path $\bar\gamma_\delta=\bar{h}(e(\delta))\cap F_{v(\delta)}$
		from $\bar\nu_{v(\delta)}$ to $\bar\nu_{e(\delta)}$ for each end-of-edge $\delta\in\widetilde{\edge}(\cfggraph)$, 
		where $v(\delta)\in\vrtx(\cfggraph)$ is the vertex adjacent to $\delta$, and $e(\delta)\in\edge(\cfggraph)$
		is the edge carrying $\delta$. We may also assume that the $\bar\nu_v$'s and 
		$\bar\nu_e$'s are vertices of the cubulated $F_v$, and hence they 
		miss the bind-curves $\binds$.
		By picking a direction of all the edges of $\cfggraph$, it makes sense to speak of the algebraic 
		intersection number of any $\bar\gamma_\delta$ with the bind-curves $\binds$, which will be {\it ad hoc}
		written as $\bar{\xi}(\bar\gamma_\delta)$.
		For every vertex $v\in\vrtx(\cfggraph)$, let:
			$$\tilde\nu_v\,=\,(\widehat\nu_v,\,\vec0),$$
		be the point in $\widehat{X}_v=\widehat{F}_v\times\RR^{n+1}_{(i(v))}$, where
		$\widehat\nu_v\in\widehat{F}_v$ is a lift of $\bar\nu_v\in F_v$.
		Thus the lift of $\bar{h}(\cfggraph)\cap F_v$ into $\widehat{F}_v$ based at $\widehat\nu_v$
		gives rise to the $\widehat\nu_{\delta}\in\widehat{F}_v$ for every end-of-edge $\delta\in\widetilde\edge(v)$ adjacent to $v$.
		Let:
			$$\tilde\nu_{\delta}\,=\left(\widehat\nu_\delta,\,\vec\Delta_{(i(v))}(\bar\xi(\bar\gamma_{\delta}))\right),$$
		be the further lift $\bar\nu_{e(\delta)}$ into $\widehat{X}_v$,
		where $\vec\Delta_{(i(v))}:\RR\to\RR^{n+1}_{(i(v))}$ is the diagonal inclusion, namely,
		$\vec\Delta_{(i(v))}(t)=(t,\cdots,t)$ for any $t\in\RR$.
		This naturally gives rise to a lift of $\bar{h}(\cfggraph)\cap F_v$ into $\widehat{F}_v\times\RR^{n+1}_{(i(v))}$ based
		at $\tilde\nu_v$, for instance, using the linear extension on the
		$\RR^{n+1}_{i(v)}$ factor. We denote the corresponding images in the quotient as:
		$$\nu_v\in X_v,\textrm{ and }\nu_\delta\in \partial_\delta X_{v(\delta)},$$ 
		which are all vertices in the cubulation.
		
		Glue up the $X_v$'s matching the reference points on the boundary. We denote the
		gluing by:
			$$\Phi_\delta: \partial_\delta X_{v(\delta)}\to\partial_{\bar\delta} X_{v(\bar\delta)},$$
		for every end-of-edge $\delta\in\widetilde\edge(\cfggraph)$,
		satisfying $\Phi_{\bar\delta}=\Phi_{\delta}^{-1}$, and $\Phi_\delta(\nu_\delta)=\nu_{\bar\delta}$.
		Thus we obtain a cube complex $X$,
		and in addition, a lifted $h:\cfggraph\to X$ such that $r\circ h:\cfggraph\to \cfggraph$ is the identity, where
		$r:X\to\cfggraph$ is the natural collapsing.
		
		\begin{lemma} 
			Different choices of $\bar{h}:\cfggraph\to F$
			and the $\widehat\nu_v$'s in $\widehat{F}$ yield the same cube complex
			$X$ up to combinatorial isometry.
		\end{lemma}
		
		\begin{proof} 
			Suppose we choose a new lift of
			the configuration graph $\bar{h}':\cfggraph\to F$ with a new collection
			of lifted base-points $\widehat\nu'_v\in\widehat{F_v}$. Denote the gluing
			map with respect to the new 
			reference points by $\Phi_\delta':\partial_\delta X_{v(\delta)}\to\partial_{\bar\delta} X_{v(\bar\delta)}$.
			We need to find a combinatorial isometry $\tau_v:X_v\to X_v$ for each $v\in\vrtx(\cfggraph)$,
			such that the following commutative diagram holds for every end-of-edge $\delta\in\widetilde\edge(\cfggraph)$:
				$$\begin{CD}
					\partial_\delta X_{v(\delta)}@>\Phi'_\delta>> \partial_{\bar\delta} X_{v(\bar\delta)}\\
					@V \tau_{v(\delta)} VV	@V \tau_{v(\bar\delta)} VV\\
					\partial_\delta X_{v(\delta)}@>\Phi_\delta>> \partial_{\bar\delta} X_{v(\bar\delta)}.
				\end{CD}$$
			
			To obtain the $\tau_v$'s, let each $\widehat\beta_v$ be an immersed path in $\widehat{F}_v$
			from $\widehat\nu_v$ to $\widehat\nu_v'$, and let $\bar\beta_v$ be
			the projected immersed path in $F_v$. We define:
				$$\tau_v:X_v\to X_v,$$
			to be the combinatorial translation 
			induced from the integral translation on $\widehat{X}_v$ given by:
				$$\tilde\tau_v.(w,\,\vec\lambda_{(i(v))})\,=\,\left(w,\,\vec\lambda_{(i(v))}+
					\vec\Delta_{(i(v))}(\bar\xi(\bar\beta_v))\right).$$ 
			
			It is straightforward to check that these $\tau_v$'s satisfy the diagram above.
			To be precise, let each 
			$\bar\beta_e$ be an immersed path in $z_e$ from $\bar\nu_e$ to $\bar\nu_e'$,
			and let $\widehat\beta_\delta$ be the lifted path in $\widehat{F}_v$
			based at $\widehat\nu_\delta$, where $\delta$ is the end-of-edge adjacent to $v$ and
			carried by $e$. We further lift $\widehat\beta_\delta$ as a path $\tilde\beta_\delta$
			in $\widehat{X}_v$ based at $\tilde\nu_v$, so that the $\RR^{n+1}_{(i(v))}$
			component of the other end-point $\tilde\nu^*_\delta$ 
			is $\vec\Delta_{(i(v))}(\bar\xi(\bar\beta_e))$. Let $g\in\pi_1(F_v)$ be the element
			represented by the loop obtained from joining the consecutive directed paths $-\bar\beta_e$,
			$-\bar\gamma_\delta$, $\bar\beta_v$ and $\bar\gamma'_{\delta}$. We regard $g$ as in $\pi_1(J_v)$ via
			the natural inclusion induced by $F_v\subset J_v$. Then from the construction,
			$\vec\xi_{(i(v))}([g])=\vec\Delta(\bar\xi([g]))$. Moreover, we have:
				$$\tilde\tau_v.\,\tilde\nu'_\delta\,=\,g.\,\tilde\nu^*_\delta.$$
			This means in the quotient space $X_v$, $\tau_{v(\delta)}.\nu'_\delta=\nu^*_\delta$.
			However, from the construction we also have $\Phi_\delta(\nu^*_\delta)=\nu^*_{\bar\delta}$.
			Hence it is verified that $\Phi_\delta\circ\tau_{v(\delta)}=\tau_{v(\delta)}\circ\Phi'_\delta$.
		\end{proof}
		
		Hence the cube complex $X$ is canonically determined by the multitwist $(F,\sigma)$ and
		the cut-bind system $(\cuts,\binds)$.
		
		\begin{lemma}\label{Xproduct} 
			The cube complex $X$ is homeomorphic to $M_\sigma\times\RR^n$ with a 
			canonical isomorphism $\pi_1(X)\cong\pi_1(M_\sigma)$. 
		\end{lemma}
		
		\begin{proof}
			Under the natural inclusion $F\subset M_\sigma$, the reference points also gives rise to
			natural reference points for the JSJ pieces and JSJ tori of $M_\sigma$. Thus the 
			caonical isomorphisms $\pi_1(X_v)\cong\pi_1(M_v)$ induces a canonical isomorphism
			$\pi_1(X)\cong\pi_1(M_\sigma)$ using the graph-of-groups decomposition.
			To see $X$ is homeomorphic to $M_\sigma\times\RR^n$, we take an $n$-plane
			foliation of $X_v$ for each $v\in\vrtx(\cfggraph)$ as in the proof of Lemma \ref{X_Pproduct}.
			For any $e\in\edge(\cfggraph)$, the foliations on the quotient of $\RR^{n+2}$ modulo $\pi_1(T_e)$
			induced from the two adjacent pieces are clearly isotopic in the space of
			(affine) $n$-plane foliations. $X$ is topologically
			a trivial $n$-plane bundle over $M_\sigma$. Hence $X$ is homeomorphic to $M_\sigma\times\RR^n$.
		\end{proof}

	\subsection{Hyperplanes in $X$}\label{Subsec-hyperplanes} 
		In this subsection, we study the
		cubical geometry of the hyperplanes of $X$. This will prepare us for the proof
		of the virtual specialness of $X$.
		
		There are $n+3$ types hyperplanes in $X$, one called the cut-type and
		the others called the bind-types, indexed by $j\in\{1,2,\cdots,n+2\}$. With the notations
		from Subsection \ref{Subsec-cubCplxX}, an edge of the cubulation of 
		$F$ is said to have the \emph{cut-type} (resp. the \emph{bind-type})
		if it is dual to a cut-curve (resp. a bind-curve). This also gives rise to
		types of edges in any covering space of $F$ with the induced cubulation. 
		In the subcomplex $X_v$ corresponding to 
		a vertex $v\in \vrtx(\cfggraph)$, a \emph{cut-type} edge is an edge
		projected from a cut-type edge parallel to $\widehat{F}$ in $\widehat{F}\times\RR^{n+1}_{(i(v))}$; and 
		a \emph{bind-type} edge of \emph{index} $j$, 
		is either an edge projected from a bind-type edge 
		parallel to $\widehat{F}$ in $\widehat{F}\times\RR^{n+1}_{(i(v))}$, if $j$ is the vertical index $i(v)$,
		or an edge projected from an edge parallel to $\RR^{n+1}_{(i(v))}$, if $j$ is a horizontal index. 
		By our construction, it is clear that every hyperplane is dual to exactly
		one type of edges. This gives rise to $n+3$ types of hyperplanes. 
		
		\begin{lemma}\label{geom-cutType} 
			The cut-type hyperplanes in $X$ are in natural bijection to
			the cut-curves on $F$, namely, the components of $\cuts$. 
			Moreover, for every cut-type hyperplane is two-sided,
			homeomorphic to $T^2\times\RR^n$. Here $T^2$ denotes the torus.
		\end{lemma}
		
		\begin{proof} 
			Let $z\subset\cuts$ be a cut-curve, and $v\in\vrtx(\cfggraph)$ be one of the at most two vertices 
			for which $z\subset F_v$ (possibly on the boundary). Then for any copy of the universal covering $\widehat{z}$ of $z$
			in $\widehat{F}_v$, the stabilizer of $\widehat{z}\times\RR^{n+1}_{(i(v))}$ is 
			clearly a conjugate of $\pi_1(T_z)$ in $\pi_1(J_v)$, where $T_z\subset M_\sigma$ denotes the mapping torus $M_{\sigma|z}$
			of $\sigma$ restricted to $z$, which is a torus. Thus the quotient hyperplane $H_z\subset X$ is two-sided, homeomorphic 
			to $T^2\times\RR^n$. Moreover, different choices of the copies projects to the same $H_z$, and every cut-type
			hyperplane $H\subset X$ is naturally some $H_z$ by the definition. This gives the natural bijection between cut-type
			hyperplanes of $X$ and the cut-curves on $F$.
		\end{proof}
		
		We denote the union of all the cut-type hyperplanes in $X$ as $\hypps_\cuts$, and denote the component 
		of $\hypps_\cuts$ corresponding to the cut-curve $z\subset\cuts$ as $H_z$.
		It is also clear that the components of $X-\hypps_\cuts$ are in natural
		bijection to the pairs-of-pants on $F$, namely, the components of $\pants=F-\cuts$; moreover,
		each component is homeomorphic to the product of a pair-of-pants with $S^1\times\RR^n$. 
		We denote:
			$$X_\pants\,=\,X-\hypps_\cuts,$$
		and denote the component corresponding to the pair-of-pants $P\subset\pants$ as $X_P$. 
		This induces
		a graph-of-spaces decomposition of $X$, whose dual graph:
			$$\cutgraph\,=\,\cutgraph(\pants,\cuts),$$
		is naturally isomorphic to the graph dual to the pants decomposition $\pants=F-\cuts$ of $F$.
		This decomposition also naturally agrees with the decomposition of $M_\sigma$:
			$$\pieces_{\pants}\,=\,M_\sigma-\tori_\cuts,$$ 
		where $\tori_\cuts$ is the disjoint union of the mapping tori $T_z=M_{\sigma|z}$ of $\sigma$
		restricted to each cut-curve $z\subset\cuts$, and $\pieces_{\pants}$ is the disjoint union of
		of the mapping tori $J_P=M_{\sigma|P}$ of $\sigma$ restricted to each pair-of-pants $P\subset\pants$.
		As the pants decomposition is subordinate to $(F,\sigma)$, $\cutgraph$ refines
		the configuration graph $\cfggraph$, in the sense that there is an obvious
		natural collapse $\cutgraph\to\cfggraph$. We often write $v(P)$ for the
		vertex $v\in\cfggraph$ such that $F_v$ contains $P$. We call $X_\pants=X-\hypps_\cuts$ the
		\emph{cut decomposition} of $X$.
		
		For every index $j\in\{1,2,\cdots,n+2\}$, we denote the union of all the bind-type hyperplane 
		of index $j$ as:
			$$\hypps^j\subset X.$$
			
		\begin{lemma}\label{geom-bindType} 
			For every index $j\in\{1,2,\cdots,n+2\}$,
			there are only finitely many index-$j$ bind-type hyperplanes in $X$. 
			Moreover, index-$j$ bind-type hyperplanes are all two-sided, canonically oriented, 
			so that $\hypps^j$ is dual to $\xi^j$ in $H^1(X)\cong H^1(M_\sigma)$.
		\end{lemma}
		
		\begin{proof} 
			Because the bind-curves $\binds$ are canonically oriented (Definition \ref{cbs}),
			and because edges in $\RR^{n+2}$ are naturally directed by the basis vectors,
			it is clear that $\hypps^j$ is two-sided and canonically oriented.
			
			To see there are only finitely many bind-type hyperplanes in $X$, observe that
			every cut-type hyperplane, as an $(n+2)$-dimensional cube complex by itself,
			has only finitely many sub-hyperplanes. 
			In fact, let $H_z\subset X$ be the cut-type hyperplane corresponding to the
			cut-curve $z\subset\cuts$. With the notations in the proof
			of Lemma \ref{geom-cutType}, the element in $\pi_1(T_z)$
			represented by the cut-curve $z$ acts on $\widehat{z}\times\RR^{n+1}_{i(v)}$ 
			as a translation by
			$\vec\xi(z)=(\xi^1([z]),\cdots,\xi^{n+2}([z]))=\vec\Delta(\bar\xi([z]))$,
			which is nontrivial on every component. Thus for every $j\neq i(v)$ in $\{1,\cdots,n+2\}$,
			there are only finitely many distinct
			sub-hyperplanes in $H_z$, which are projected by a parallel copy of $j$-th coordinate sub-hyperplane
			$\RR^n_{(i(v),j)}$, namely, the intersection of $\RR^{n+1}_{(i(v))}$ and $\RR^{n+1}_{(j)}$. 
			This means $H_z$ has only finitely many sub-hyperplanes.
			This finiteness implies the finiteness of bind-type hyperplanes in $X$, 
			because every bind-type hyperplane must intersect
			at least one $H_z$ in a (nonempty) sub-hyperplane by our construction. 
			
			Counting the algebraic intersection of $\hypps^j$ with any representative cycle
			of $H_1(X)$ induces a dual cohomology class in $H^1(X)\cong {\Hom(H_1(X),\ZZ)}$.
			To see that $\hypps^j$ is dual to $\xi^j$, 
			observe that the restriction of $\xi^j$ to each $H^1(X_v)$ is dual to $\hypps^j\cap X_v$. 
			This follows immediately from the construction of $X_v$. Furthermore,
			it is also clear from our canonical gluing that for any lift $h:\cfggraph\to X$
			of the configuration graph, $\xi^j$ evaluates on the image of $H_1(\cfggraph)$
			precisely by counting the algebraic intersection number
			of any representative cycle with $\hypps^j$. Note that the image of
			$H_1(\cfggraph)$ and all the $H_1(X_v)$'s generate $H_1(X)$. Therefore, 
			$\hypps^j$
			is dual to $\xi^j\in H^1(X)$.
		\end{proof}
		
		While they are complicated in general, bind-type hyperplanes are locally easy to understand.
		Note that for any bind-type hyperplane $H$ in $X$, and
		for any pair-of-pants $P\subset\pants$, the components of $H\cap X_P$ are all hyperplanes
		in $X_P$, with the bind-type index known as that of $H$. The \emph{vertical index} $i(P)$
		with respect to $P$ is known as the vertical index $i(v)$ with respect to
		the unique $v\in\vrtx(\cfggraph)$ such that $P\subset F_v$; 
		and the other indices are the \emph{horizontal indices} with respect to $P$.
		A bind-type hyperplane in $X_P$ of index $j$
		is said to be \emph{vertical} if its index 
		$j$ equals $i(P)$; it is said to be \emph{horizontal}, if
		$j$ differs from $i(P)$.
		
		\begin{lemma}\label{geom-vertical} 
			For every pair-of-pants $P\subset\pants$,
			the vertical hyperplanes in $X_P$ are in natural bijection to the 
			bind-arcs in $P$, namely, the components of 
			$P\cap\binds$. Moreover, each vertical hyperplane in $X_P$
			is homeorphic to $A^2\times\RR^n$,
			with the two boundary components lying on two distinct components of $\partial X_P$.
			Here $A^2$ denotes the compact annulus.
		\end{lemma}
		
		\begin{proof} 
			Note that the universal covering $\widehat{X}_P$
			is $\widehat{P}\times\RR^{n+1}_{(i(P))}$, where $\widehat{P}$ is the universal
			covering of $P$. For every vertical hyperplane $Q\subset X_P$,
			every component $\hat{Q}\subset\widehat{X}_P$
			of its preimage is the product
			$\hat{w}\times\RR^{n+1}_{(i(P))}$, where $\hat{w}\subset\widehat{P}$ is
			a preimage-component of a unique bind-arc $w\subset P\cap\binds$. 
			The stabilizer of $\hat{Q}$ in $\pi_1(J_P)$ is $\pi_1(f_{v(P)})$.
			This implies that $Q$ is homeomorphic 
			to $A^2\times\RR^n$. 
			Moreover, different choices of the preimage-component $\hat{Q}$ projects
			to the same $Q\subset X_v$, so $Q\mapsto w$ is well-defined, sending
			every vertical hyperplane in $X_P$ to a bind-arc $w$ on $P$. It is straightforward
			to check that this defines a bijection between these two kinds of objects.
		\end{proof}
		
		For every pair-of-pants $P\subset\pants$,
		and for every horizontal index $j$, we denote the restriction
		of $\xi^j$ to $J_P$ as $\xi^j_P\in H^1(J_P)$, and
		the restriction of $\xi^j$ to $T_z$ as $\xi^j_z\in H^1(T_z)$.
		Remember also from Subsection \ref{Subsec-setupData}
		that $\bar\xi\in H^1(F)$ is the cohomology class dual to the bind-curves $\binds$.
		We denote the restriction of $\bar\xi$ to $P$ as $\bar\xi_P\in H^1(P)$. For any nonzero
		integer $d$, $\bar\xi_P$ modulo $d$ lives in $H^1(P;\,\ZZ_d)$, and it induces
		a dual cyclic covering of $P$ corresponding to the kernel of $\pi_1(P)\to
		H_1(P)\to\ZZ_d$, whose degree is the greatest common divisor of 
		$d$ and the divisibility $\mathrm{div}(\bar\xi_P)$
		of $\bar\xi_P$. Recall that the \emph{divisibility} $\mathrm{div}(\xi)$ 
		of a nontrivial element $\xi$ in a free abelian group $V$
		is the largest positive integer by which $\xi$ is divisible in $V$. 
		
		\begin{lemma}\label{geom-horizontal} 
			For every pair-of-pants $P\subset\pants$,
			and for every horizontal index $j$, there are exactly
			$\mathrm{div}(\xi^j_P)$ index-$j$ horizonal hyperplanes in $X_P$. Moreover,
			these hyperplanes are all parallel, each homeomorphic
			to $\tilde{P}^j\times\RR^n$, where $\tilde{P}^j$ is the finite cyclic covering
			of $P$ dual to $\bar\xi_P$ modulo $\xi^j([f_{v(P)}])$.
			In particular, for every cut-curve $z\subset\cuts$ on $\partial P$,
			each horizontal hyperplane has exactly $\mathrm{div}(\xi^j_z)\,/\,\mathrm{div}(\xi^j_P)$ 
			boundary components on the corresponding component of $\partial X_P$.
		\end{lemma}
		
		\begin{proof} 
			Let $Q\subset X_P$ be the horizontal hyperplane of index $j$. If $X_v$ is the
			piece containing $X_P$, then the construction of $X_v$ implies that $X_P$ is the 
			quotient of its universal covering $\widehat{X}_P\cong\widehat{P}\times\RR^{n+1}_{(i(P))}$
			by the action $\pi_1(J_P)$ defined in the same manner, and $Q$ is projected from
			from $\widehat{P}\times\RR^n_{(i(P),j)}$,
			where $\RR^n_{(i(P),j)}$ for the intersection $\RR^{n+1}_{(i(P))}\cap\RR^{n+1}_{(j)}$.
			The stabilizer:
				$$G_P^j\,=\,{\rm Stab}_{\pi_1(J_P)}\,(\widehat{P}\times\RR^n_{(i(P),\,j)}),$$ 
			consists of elements $g\in\pi_1(P)$ such that $\xi^j([g])=0$. 
			Note that the canonical Seifert-fibration induces
			the commutative diagram of groups, where the rows are short exact sequences:
				$$\begin{CD}
					1	@>>>	\pi_1(f_v)	@>>>	\pi_1(J_P)	@> p_{v}  >>	\pi_1(P)	@>>>	1\\
						@.	@V\cong VV	@VVV	@VVV	@.\\
					0	@>>>	H_1(f_v)	@>>> 	H_1(J_P)	@> p_{v*} >>	H_1(P)  	@>>>	0.
				\end{CD}$$
			Regarding $\xi^j_P$ naturally as a linear functional in
			$\mathrm{Hom}(H_1(J_P),\,\ZZ)$, the stabilizer
			$G_P^j$ is precisely the preimage in $\pi_1(J_P)$
			of $\mathrm{Ker}(\xi^j_P)$ with respect to
			the abelianization $\pi_1(J_P)\to H_1(J_P)$. 
			Because $\xi^j([f_v])$ does not vanish, (Lemma \ref{cohomologyClasses}), 
			$G_P^j\cap\pi_1(f_v)$ is trivial, so $G_P^j$ embeds into $\pi_1(P)$ under
			$p_v$. This means $G_v^j$ acts properly on the $\widehat{P}$ factor of $\widehat{P}\times\RR^n_{(i(P),\,j)}$,
			so $Q$ is homeomorphic to $\tilde{P}^j\times\RR^n$, where $\tilde{P}^j$ is the normal
			covering of $P$ corresponding to the subgroup $p_v(G_P^j)$ of $\pi_1(P)$. Moreover,
			\begin{eqnarray*}
				\pi_1(P)\,/\,p_v(G_P^j)	&\cong&	H_1(P)\,/\,p_{v*}(\mathrm{Ker}(\xi^j_P))\\
				                     	&\cong&	H_1(J_P)\,/\,\left(H_1(f_v)\oplus\mathrm{Ker}(\xi^j_P)\right)\\
				                     	&\cong&	\mathrm{Im}(\xi^j_P)\,/\,\left(\xi^j_P([f_v])\ZZ\right)\\
                     					&\cong& \mathrm{Im}(\bar\xi_P)\,/\,\left(\xi^j_P([f_v])\ZZ\right).
            \end{eqnarray*}
            From this we see that
            $\tilde{P}^j$ is a finite cyclic covering of $P$ dual to $\bar\xi_P$ modulo $\xi^j_P([f_v])$.
            
            To see there are exactly $\mathrm{div}(\xi^j_P)$ copies of $Q$ in $X_P$, note that the parallel
            copies of $\widehat{P}\times\RR^n_{(i(P),\,j)}$ are parametrized as $\ZZ$ by the $j$-th coordinate
            of $\RR^{n+2}$. The action of $\pi_1(J_P)$ on the spaces of these copies is exactly the translation
            group generated by a translation of distance $\mathrm{div}(\xi^j_P)$.
            
            Finally, the number of boundary components of $Q$ on the boundary component $H_z\subset \partial X_P$ 
            corresponding to $z\subset \partial P$ is the index of the image $\pi_1(z)$ in $\pi_1(P)\,/\,p_v(G_P^j)$. In other words,
            it is the index of the subgroup of $\mathrm{Im}(\xi^j_P)$
            generated by $\xi^j_z([z])$ and $\xi^j_z([f_v])$. Note  $\mathrm{Im}(\xi^j_P)$
            is identified as the subgroup $\mathrm{div}(\xi^j_P)\ZZ$ of $\ZZ$.
            Since $[z]$ and $[f_v]$ generates $H_1(T_z)$, $\mathrm{div}(\xi^j_z)$ is just the greatest
            common divisor of $\xi^j_z([z])$ and $\xi^j_z([f_v])$. Hence the index
            equals $\mathrm{div}(\xi^j_z)\,/\,\mathrm{div}(\xi^j_P)$.
		\end{proof}
		
		In order to capture the global geometry of bind-type hyperplanes in $X$, it is useful to consider 
		its decomposition induced from the cut decomposition of $X$.
		Specifically, let $H\subset X$ be a bind-type hyperplane. There is a naturally associated
		graph $\cutgraph_H$, whose vertices correspond to the components of $H\cap X_\pants$,
		and whose edges correspond to the components of $H\cap \hypps_\cuts$.
		Hence there is a naturally associated map $\iota_H:\,\cutgraph_H\to\cutgraph$
		between graphs, satisfying the following commutative diagram: 
			$$\begin{CD}
				H@>\subset>> X\\
				@VVV @VVV\\
				\cutgraph_H@>\iota_H>>\cutgraph,
			\end{CD}$$
		where the vertical maps are natural collapses.
		Note also that $\cutgraph_H$ is a finite graph by the proof of Lemma \ref{geom-bindType}.
		
		\begin{lemma}\label{iotaH} 
			For a bind-type hyperpane $H$ of $X$, 
			the associated map $\iota_H:\cutgraph_H\to\cutgraph$ is
			an immersion (i.e.~a local isometry) if and only if
			no component of any $H\cap X_P$ has two boundary components 
			lying on the same component of $\partial X_P$; and $\iota_H$ is
			an embedding if and only if $H\cap X_P$ is connected (possibly empty) for every $X_P$.
		\end{lemma}
		
		\begin{proof} This follows immediately from the definition.\end{proof}
		
		We discuss which pathologies may occur to hyperplanes in $X$. In other words,
		we study the position relationship of hyperplanes near vertices. 
		There are two types of vertices in $X$, called the octagon type and the band type.
		An \emph{octagon-type} vertex comes from the center of the octagon in a pair-of-pants, while
		a \emph{band-type} vertex comes from the center of a band in a pair-of-pants, (cf.~Figure \ref{figCubPants2D}).
		
		\begin{lemma}\label{positionsBasic} 
			The following statements are true.
			\begin{enumerate}
				\item Every hyperplane in $X$ is two-sided with no self-intersection;
				\item Any two distinct hyperplanes of the same type do not intersect, and hence do not inter-osculate;
				\item Any two bind-type hyperplanes of distinct indices do not osculate, and hence do not inter-osculate.	
			\end{enumerate}
		\end{lemma}
		
		\begin{proof} 
			The first statement has been proved in Lemma \ref{geom-cutType},
			Lemma \ref{geom-bindType}. The second statement is true by the definition
			of types. The third statement is true by the definition of osculation. In fact, it is
			straightforward to check at each octagon-type vertex and at each band-type vertex, that two edges
			dual to two bind-type hyperplanes of distinct indices, respectively, always share a common square.
		\end{proof}
		
		There remains to be three kinds of possible pathologies:
		\begin{itemize}
			\item Self-osculation of cut-type hyperplanes;
			\item Self-osculation of bind-type hyperplanes;
			\item Inter-osculation between cut-type hyperplanes and bind-type hyperplanes.
		\end{itemize}
		To be more illustrative, let us name the edges in Figure \ref{figCubPants2D} adjacent to the octagon center
		by the compass directions, such as the east, the northwest, etc.
		Self-osculation of cut-types occurs at every octagon-type vertex, as the north and the south
		are dual to the same cut-type hyperplane. Self-osculation of bind-types occurs at an octagon-type
		vertex, when the northwest (resp. southwest) and the southeast (resp. northeast)
		are dual to the same bind-type hyperplane. 
		Inter-osculation between mixed types may also occur. For instance,
		if the northwest and the northeast are dual to the same
		bind-type hyperplane, this hyperplane will inter-osculate
		both the cut-type hyperplanes dual to the east and the west.

	\subsection{Virtual specialness of $X$}\label{Subsec-virSpX} 
		In this subsection, we show the cube complex $X$ 
		constructed in Subsection \ref{Subsec-cubCplxX} is virtually special, namely, such that a finite covering
		of $X$ is a special cube complex (Definition \ref{spCubeCplx}).
		
		Note that for any covering $\kappa:\tilde{X}\to X$, it makes sense to speak of 
		the types of hyperplanes and vertices in $\tilde{X}$, according to the type of the underlying image.
		There will also be an induced cut decomposition of $\tilde{X}$,
		over a graph $\tilde{\cutgraph}$ whose vertices correspond to the components of the
		$\kappa^{-1}(X_\pants)$, and whose edges correspond to the components of
		$\kappa^{-1}(\hypps_\cuts)$. Moreover, for any bind-type hyperplane $\tilde{H}$ of $\tilde{X}$, there is also an
		associate decomposition graph $\tilde{\cutgraph}_{\tilde{H}}$, which is a finite graph,
		and there is an associated map $\iota_{\tilde{H}}: 
		\tilde\cutgraph_{\tilde{H}}\to\tilde\cutgraph$, for which the conclusion of Lemma \ref{iotaH} still holds.
		
		Our strategy is as follows. For every index $j\in\{1,2,\cdots,n+2\}$,
		pass to a finite cyclic covering $\tilde{X}^j$ of $X$, so that 
		there is no self-osculation of cut-type hyperplanes in an $X_P$ with
		the vertical index $i(P)$ equal to $j$, 
		and that every bind-type hyperplane of
		index $j$ in $\tilde{X}^j$ induces an immersion between the decomposition graphs. 
		Then we pass to a further finite regular covering $\ddot{X}^j$
		induced by a finite regular covering of the decomposition graph, 
		so that every bind-type hyperplane of index $j$ in $\ddot{X}^j$
		induces an embedding between the decomposition graphs. This will eliminate self-osculations
		of index-$j$ bind-types as well as inter-osculations between cut-types and index-$j$
		bind-types. It follows that a common finite regular covering $\ddot{X}$
		of all these $\ddot{X}^j$'s will be special.
		
		\medskip\noindent\textbf{Step 1. The finite cyclic covering $\tilde{X}^j$.}
		Let $j\in\{1,2,\cdots,n+2\}$ be an index. Note that for every cut-curve $z\subset\cuts$,
		the restricted cohomology class
		$\xi^j_z\in H^1(T_z)$ is nontrivial as it evaluates nontrivially on $[z]\in H_1(T_z)$, 
		(Lemmas \ref{barXi}, \ref{cohomologyClasses}). Let
		the positive integer $l^j$ be the least common multiple of all
		the divisibilities $\mathrm{div}(\xi^j_z)$ as $z$ runs over the all components
		of the cut-curves $\cuts$.
		Let:
			$$\tilde\kappa^j:\,\tilde{X}^j\to X,$$
		be the finite cyclic covering corresponding to the kernel of the
		the composed homomorphism:
			$$\begin{CD}\pi_1(M_\sigma)	@>[\cdot]>>	H_1(M_\sigma)	@>{\xi^j}>>	\ZZ	@>\bmod l^j >>	\ZZ_{l^j}.\end{CD}$$ 
		It is clear that the covering degree is $l^j\,/\,\mathrm{div}(\xi^j)$.
		We denote the induced decomposition graph of $\tilde{X}^j$ as $\tilde\cutgraph^j$.
		
		\begin{lemma}\label{noOsculateCut} 
			For any pair-of-pants $P\subset\pants$
			with the vertical index $i(P)$ equal to $j$,
			and for any vertex $\tilde\nu\in\tilde{X}^j$ with $\tilde\kappa^j(\tilde\nu)\in X_P$,
			no cut-type hyperplane in $\tilde{X}^j$ self-osculates at $\tilde\nu$.
		\end{lemma}
		
		\begin{proof} 
			By Lemma \ref{geom-cutType}, cut-type hyperplanes do not
			self-osculate at any band-type vertex in $X$. In fact, there are two cut-type edges
			adjacent to each band-type vertex, but they are dual to different hyperplanes.
			Thus cut-type hyperplanes do not self-osculate
			at any band-type vertex in $\tilde{X}^j$ either. 
			
			Suppose there were an octagon-type vertex $\tilde\nu\in \tilde{X}^j$, such that $\tilde\kappa^j(\tilde\nu)
			\in X_P$ for some pair-of-pants $P\subset\pants$ with the vertical index $i(P)$ equal to $j$,
			and that there are two cut-type edges adjacent to it 
			dual to the same cut-type hyperplane $\tilde{H}\subset \tilde{X}^j$.
			Pick a path on $\tilde{H}$ joining the two end-points on $\tilde{H}$ of the two edges.
			Project the loop $\tilde{c}$ formed by the path and the two edges to $X$ via $\tilde\kappa^j$, 
			then $c=\tilde{\kappa}^j(\tilde{c})$ represents an element in $H_1(X)$. Let $z\subset\cuts$ be 
			the cut-curve whose corresponding hyperplane $H_z\subset X$ equals $\tilde{\kappa}^j(\tilde{H})$,
			(Lemma \ref{geom-cutType}).
			Then by counting the algebraic intersection number of $c$ and $\hypps^j$,
			$\xi^j([c])$ modulo $\xi^j([z])$ is nontrivial. Note that the absolute value of $\xi^j([z])$
			is the divisibility of $\xi^j_z$, since $H_1(T_z)$ is generated by
			$[z]$ and $[f_v]$ while $\xi^j([f_v])=0$, (Lemma \ref{cohomologyClasses}).
			Thus $\xi^j([c])$ modulo $l^j$ is nontrivial either. Such a loop $\tilde{c}$ cannot lie in the covering space
			$\tilde{X}^j$, so we reach a contradiction.
		\end{proof}
		
		\begin{lemma}\label{immersionIotaH} 
			For every index-$j$ bind-type hyperplane $\tilde{H}\subset \tilde{X}^j$,
			the associated map between the decomposition graphs:
				$$\iota_{\tilde{H}}:\,\tilde\cutgraph^j_{\tilde{H}}\to \tilde\cutgraph^j,$$
			is an immersion.
		\end{lemma}
		
		\begin{proof} 
			By Lemma \ref{iotaH}, we must show that for any component $\tilde{X}_P\subset \tilde{X}^j$
			of the preimage $(\tilde\kappa^j)^{-1}(X_P)$,  no two boundary components
			of $\tilde{Q}=\tilde{H}\cap\tilde{X}_P$ lie in the same component of $\partial\tilde{X}_P$. 
			
			If $\tilde{Q}$ is vertical in $\tilde{X}_P$, this is automatically
			true by Lemma \ref{geom-vertical}.
			
			If $\tilde{Q}$ is horizontal in
			$\tilde{X}_P$, suppose there were some component $\tilde{H}\subset\partial\tilde{X}_P$ containing two
			distinct boundary components $\tilde{U},\tilde{U}'\subset\tilde{Q}$. Let $z\subset\cuts$ be the cut-curve
			so that $H_z\subset X$ equals $\tilde\kappa^j(\tilde{H})$.
			Because $\tilde\kappa^j$ is the cyclic covering induced by $\xi^j$ modulo $l^j$,
			and because the union $\hypps^j\subset X$ of index-$j$ bind-type hyperplanes is dual to
			$\xi^j$ by our construction, (Lemma \ref{geom-bindType}), 
			$\tilde{X}^j$ can be obtained by cyclically gluing 
			$l^j\,/\,\mathrm{div}(\xi^j)$ copies of 
			$X-\hypps^j$. In particular, $\tilde\kappa^j$ embeds $\tilde{Q}^j$ into $X_P$,
			so the projected images of $\tilde{U}$ and $\tilde{U}'$ in $X_v$ are different as well. 
			Pick a path in $\tilde{Q}$
			with the two end-points $\tilde\nu,\tilde\nu'$ on $U,U'$, respectively, and
			pick another path in $\tilde{H}$ joining $\tilde\nu, \tilde\nu'$.
			Project the loop $\tilde{c}$ formed by these two paths into $X_P$, 
			so that $c=\tilde{\kappa}^j(\tilde{c})$ 
			represents an element in $H_1(X_v)$. As 
			$\tilde{U},\tilde{U}'$ have distinct projected images, 
			counting the algebraic intersection number of 
			$c$ with $\hypps^j$ yields that
			$\xi^{j}([c])$ modulo $\mathrm{div}(\xi^j_z)\,/\,\mathrm{div}(\xi^j_P)$ is nontrivial,
			(Lemma \ref{geom-horizontal}). Thus
			$\xi^j([c])\bmod l^j$ is nontrivial either, which leads to a contradiction as the loop $\tilde{c}$
			cannot lie in the covering space $\tilde{X}^j$.
		\end{proof}
		
		\medskip\noindent\textbf{Step 2. The graph-induced covering $\ddot{X}^j$.}
		We find some finite regular covering $\ddot{\cutgraph}^j\to\tilde\cutgraph^j$, so that
		for every bind-type hyperplane of index $j$, 
		all the elevations of $\iota_{\tilde{H}}$ into $\ddot{\cutgraph}$ are embeddings.
		Here by an \emph{elevation} of $\iota_{\tilde{H}}$ we mean 
		a map $\iota_{\ddot{H}}:\,\ddot{\cutgraph}^j_{\ddot{H}}\to\ddot{\cutgraph}^j$
		from a connected covering $\ddot{\cutgraph}^j_{\ddot{H}}$ of 
		$\tilde{\cutgraph}^j_{\tilde{H}}$ into $\ddot{\cutgraph}^j$,
		minimal subject to the following commutative diagram of maps:
			$$\begin{CD}
				\ddot{\cutgraph}^j_{\ddot{H}}	@>\iota_{\ddot{H}}>>	\ddot{\cutgraph}^j\\
					@VVV	@VVV\\
				\tilde{\cutgraph}^j_{\tilde{H}}	@>\iota_{\tilde{H}}>>	\tilde{\cutgraph}^j.
			\end{CD}$$
		Writing $\ddot{X}\to \tilde{X}$ for the naturally induced finite regular covering of $\tilde{X}$,
		it is clear that every elevation is exactly the induced map $\iota_{\ddot{H}}$ for some component
		$\ddot{H}$ of the preimage of $\tilde{H}$.
		
		\begin{lemma}\label{embeddingIotaH} 
			There exists a finite regular covering $\ddot{\cutgraph}^j\to\tilde\cutgraph^j$, so that
			for every bind-type hyperplane of index $j$, every elevation of $\iota_{\tilde{H}}$ 
			into $\ddot{\cutgraph}^j$ is an embedding.
		\end{lemma}
		
		\begin{proof} 
			This follows from the fact that finitely generated free groups are LERF, (\cite[Theorem 2.2]{Sc-LERF}).
			Recall from \cite{Sc-LERF} that a group $G$ is said to be \emph{locally extended residually finite} (or \emph{LERF}), if for
			every finitely generated subgroup $H$ of $G$, and for every $g\in G$ not in $H$, there is a finite-index subgroup $\tilde{G}$ of 
			$G$ containing $H$, which does not contain $g$. If $S$ is a connected
			CW complex with $\pi_1(S)\cong G$, then $G$ is LERF if and only if
			for every finitely generated covering $\kappa:\tilde{S}\to S$ and every compact subset
			$K\subset\tilde{S}$, $\kappa$ factors through
			a finite covering of $S$ in which $K$ embeds, (cf.~\cite[Lemma 1.4]{Sc-LERF}).
			
			Suppose $\tilde{H}\subset\tilde{X}^j$ is a bind-type hyperplane of index $j$, by
			Lemma \ref{immersionIotaH}, $\iota_{\tilde{H}}:\,\tilde\cutgraph^j_{\tilde{H}}\to
			\tilde\cutgraph^j$ is an immersion, so it is $\pi_1$-injective. 
			Let $\mathrm{Cov}(\iota_{\tilde{H}})\to\tilde\cutgraph^j$ be
			the covering corresponding to image of 
			$(\iota_{\tilde{H}})_\sharp:\,\pi_1(\tilde\cutgraph^j_{\tilde{H}})\hookrightarrow
			\pi_1(\tilde\cutgraph^j)$,
			(after picking a base-point). Then $\iota_{\tilde{H}}$ elevates to an embedding into
			$\mathrm{Cov}(\iota_{\tilde{H}})$. Because $\pi_1(\tilde\cutgraph^j)$ is a finitely generated free group, it
			is LERF, so $\mathrm{Cov}(\iota_{\tilde{H}})\to\tilde\cutgraph^j$
			factors through a finite covering $\Psi^{\tilde{H}}$,
			in which an elevated image $\tilde\cutgraph_{\tilde{H}}\hookrightarrow\Psi^{\tilde{H}}$ embeds. In other words,
			there is an elevation of $\iota_{\tilde{H}}$ into $\Psi^{\tilde{H}}$ which is an embedding.
			Pass to the characteristic covering $\ddot{\cutgraph}^{j,\tilde{H}}\to\tilde\cutgraph^j$, corresponding
			to the intersection of all the conjugates of $\pi_1(\Psi^{\tilde{H}})$ in $\pi_1(\tilde\cutgraph^j)$,
			then $\ddot{\cutgraph}^{j,\tilde{H}}$ is a finite regular covering into which every elevation of
			$\iota_{\tilde{H}}$ is an embedding.
			
			By Lemma \ref{geom-bindType}, there are only finitely many bind-type hyperplanes of index $j$ in $X$, so
			there are only finitely many index-$j$ bind-types in $\tilde{X}^j$ as well. 
			We take a common finite regular covering $\ddot\cutgraph^j$ 
			of all these $\ddot{\cutgraph}^{j,\tilde{H}}$. It is the covering as desired.
		\end{proof}
		
		Let:
			$$\ddot\kappa^j:\,\ddot{X}^j\to\tilde{X}^j,$$
		be the finite regular covering of $\tilde{X}^j$ naturally induced from the
		finite regular covering $\ddot{\cutgraph}^j\to\tilde\cutgraph^j$
		as in Lemma \ref{immersionIotaH}.
		
		\begin{lemma}\label{noOsculateBind} 
			No bind-type hyperplane of index $j$ self-osculates in $\ddot{X}^j$.
		\end{lemma}
		
		\begin{proof} 
			By Lemmas \ref{geom-vertical}, \ref{geom-horizontal}, no bind-type hyperplane
			in any $X_P$ self-osculates. Indeed, bind-type edges adjacent at any vertex are dual to distinct
			hyperplanes in $X_P$. Passing to the covering space, this also holds for any component
			of $(\ddot{\kappa}^j)^{-1}(X_P)$. Now let $\ddot{H}\subset\ddot{X}^j$ be any hyperplane. Suppose
			$\ddot{X}_P$ is any component of $(\ddot{\kappa}^j)^{-1}(X_P)$, for any $P\subset\pants$.
			By Lemmas \ref{iotaH}, \ref{embeddingIotaH}, $\ddot{H}$ intersects $\ddot{X}_P$ in a single hyperplane,
			if not in the empty set. 
			Thus at any vertex in $\ddot{X}_P$, at most one edge is dual to $\ddot{H}$ in $\ddot{X}^j$.
			This means $\ddot{H}$ does not self-osculate in $\ddot{X}^j$, which completes the proof.
		\end{proof}
		
		\begin{lemma}\label{noInterOsculate} 
			No bind-type hyperplane of index $j$ inter-osculates
			any cut-type hyperplane in $\ddot{X}^j$.
		\end{lemma}
		
		\begin{proof} 
			By Lemma \ref{geom-cutType}, every cut-type hyperplane of $X$ is contained in exactly
			two components of $X_\pants$. This also holds for cut-type hyperplanes in $\ddot{X}^j$. If
			a cut-type hyperplane $\ddot{S}\subset\ddot{X}$ intersect an index-$j$ bind-type hyperplane
			$\ddot{H}\subset\ddot{X}^j$. Suppose $\ddot{X}_P$ is a component of some $(\ddot\kappa)^{-1}(X_P)$,
			which is adjacent to $\ddot{S}$. It follows from Lemmas \ref{iotaH}, \ref{embeddingIotaH} that 
			the intersection $\ddot{H}\cap\ddot{X}_P$
			is a single hyperplane in $\ddot{X}_P$, and it does not osculate $\ddot{S}$,
			(Lemmas \ref{geom-vertical}, \ref{geom-horizontal}).
			This implies $\ddot{H}$ does not osculate $\ddot{S}$, which completes the proof.
		\end{proof}
		
		\medskip\noindent\textbf{Step 3. The common covering $\ddot{X}$.} 
		We take the covering:
			$$\ddot{\kappa}:\,\ddot{X}\to X,$$
		to be the common covering of all the $\ddot{X}^j$, as $j$ runs over $\{1,2,\cdots,n+2\}$.
		In other words, $\ddot{X}$ is the finite regular covering of $X$ corresponding
		to the intersection of all the subgroups $\pi_1(\ddot{X}^j)$ of $\pi_1(X)$.
		
		\begin{lemma}\label{ddotXspecial} 
			The cube complex $\ddot{X}$ is special.
		\end{lemma}
		
		\begin{proof} 
			By Lemma \ref{positionsBasic} (1), every hyperplane in $\ddot{X}$ is
			two-sided with no self-intersection. By Lemmas \ref{noOsculateCut},
			\ref{noOsculateBind} and the fact that
			$\ddot{X}$ is the common covering of all the $\ddot{X}^j$'s, 
			no hyperplane in $\ddot{X}$ self-osculates. By Lemmas \ref{positionsBasic} (2), (3),
			\ref{noInterOsculate}, there is no inter-osculation between hyperplanes in $\ddot{X}$
			either. We conclude that the cube complex $\ddot{X}$ is special.
		\end{proof}
		
		To sum up, we have shown that for a multitwist $(F,\sigma)$ with a bipartite configuration graph,
		if the current equations associated to $(F,\sigma)$ have a nondegenerate symmetric solution, then
		we can find a cut-bind system $(\cuts,\binds)$, (Subsection \ref{Subsec-setupData}), 
		and construct a canonical cube complex $X$, (Subsection \ref{Subsec-cubCplxX}). Moreover,
		$X$ is homeomorphic to $M_\sigma\times\RR^n$ for some $0\,\leq\,n\,\leq\,|\vrtx{\cfggraph}|-2$,
		(Lemma \ref{Xproduct}), and $X$ is virtually special, (Lemma \ref{ddotXspecial}). 
		
		This completes the proof of Proposition \ref{SpCubulation}.

\section{Proof of the main theorem}\label{Sec-mainProof}
	
	In this section, we summarize the proof of Theorem \ref{main-npcGM}.
	
	\begin{lemma}\label{SFCase} 
		Theorem \ref{main-npcGM} holds for aspherical compact Seifert fibered spaces.
	\end{lemma}
	
	\begin{proof} 
		A compact Seifert space is aspherical if and only if its base orbifold is
		compact and aspherical. It is virtually fibered if and only
		if either it has nonempty boundary, or that it
		is closed with vanishing Euler class of the fiber.
		Note that this is exactly the case when the Seifert fibered
		space is nonpositively curved, cf.~\cite{KL,GW}. 
		The other three statements are straightforward to check.
	\end{proof}

	\begin{lemma}\label{SolCase} 
		None of the statements in
		Theorem \ref{main-npcGM} is true if $M$ is virtually the
		mapping torus of an Anosov torus automorphism. Hence they are equivalent in this case.
	\end{lemma}
	
	\begin{proof} 
		Such an $M$ is $\Sol$-geometric, so it is not nonpositively curved,
		(\cite{GW}, \cite{Ya}). It is clearly not virtually RFRS either.
		In fact, for any sequence of finite index subgroups
		$\pi_1(M)=G_0\triangleright G_1\triangleright\cdots$ 
		with each $G_i\to G_i\,/\,G_{i+1}$ factoring through a free abelian group, 
		the intersection of all the $G_i$'s always
		contains a finite-index subgroup the commutator subgroup of $\pi_1(M)$.
		Note that the second and the third statement implies the fourth, (cf.
		Subsections \ref{Subsec-RAAG}, \ref{Subsec-RFRS}). Thus none of the statements
		in Theorem \ref{main-npcGM} holds for such an $M$.
	\end{proof}
		
	\begin{lemma}\label{reduction} 
		If the conclusion of Theorem \ref{main-npcGM} holds
		for the mapping torus $M_\sigma$ of every multitwist $(F,\sigma)$ with a bipartite 
		configuration graph, it holds for any nontrivial compact graph manifold as well.
	\end{lemma}
	
	\begin{proof}
		The bounded case reduces to the closed case. In fact, if $M$ is a compact graph manifold
		with nonempty boundary, it follows from \cite{Le} that
		$M$ supports a nonpositively curved Riemannian metric with totally geodesic
		boundary. Then the double $W$ of $M$ along the boundary is closed and 
		nonpositively curved. Assuming that Theorem \ref{main-npcGM} holds
		in the closed case, then $\pi_1(W)$ is virtually specially cubulated, and is virtually a subgroup
		of a right-angled Artin group, and is virtually RFRS. As $\pi_1(M)$ may
		be regarded as a subgroup of $\pi_1(W)$, the same 
		holds for $\pi_1(M)$.
		
		The closed case reduces to the fibered case. This is because all the statements of
		Theorem \ref{main-npcGM} are virtual properties, and each
		of them implies that the manifold is virtually fibered, (Theorems \ref{Thm-Sv},
		\ref{Thm-HW}, \ref{Thm-Ag}). Furthermore, Lemmas \ref{mappingTorus}, \ref{SolCase} imply that
		it suffices to prove Theorem \ref{main-npcGM} for the mapping torus
		of multitwists with a bipartite 
		configuration graphs.
	\end{proof}
	
	After these reductions, we are ready to prove Theorem \ref{main-npcGM}.
	
	\begin{proof}[{Proof of Theorem \ref{main-npcGM}}] 
		We reduced to the case for nontrivial graph manifolds
		by Lemma \ref{SFCase}. Moreover,
		by Lemma \ref{reduction}, it suffices to assume
		that $M$ is the mapping torus $M_\sigma$ of some multitwist $(F,\sigma)$ with a bipartite configuration graph.
		
		To see (1) $\Rightarrow$ (2), note the nonpositive curving implies
		a nondegenerate symmetric solution of the BKN equations associated to $(F,\sigma)$, (Theorem \ref{Thm-BS}).
		By Proposition \ref{nsSolutions}, the current equations virtually have a nondegenerated symmetric solution
		as well. After passing to a finite covering, we can apply Proposition \ref{SpCubulation} to construct a special cube
		complex homotopy equivalent to $M_\sigma$.
		
		To see (2) $\Rightarrow$ (3), we apply Theorem \ref{Thm-HW}. Note that $\pi_1(M_\sigma)$ is finitely generated,
		so we may further require the right-angled Artin group $A(\Gamma)$
		to be finitely generated, for instance, by taking subgroup $A(\Gamma')$ generated only
		by the generators of $A(\Gamma)$ that spell
		the words representing the image of the generators of $(M_\sigma)$.
		
		To see (3) $\Rightarrow$ (4), observe that finitely generated Artin groups are all virtually RFRS, (cf. Subsection
		\ref{Subsec-RFRS}).
		
		To see (4) $\Rightarrow$ (1), note the virtual RFRS-ness implies virtually a nondegenerate symmetric
		solution of the current equations, (Lemma \ref{RFRSnsSol}), and hence a nondegenerate symmetric
		solution to the BKN equations, (Proposition \ref{nsSolutions}). By Theorem \ref{Thm-BS}, $M_\sigma$ is 
		nonpositively curved. This completes the proof.
	\end{proof}

\bibliographystyle{amsalpha}


\begin{thebibliography}{}

\bibitem[Ag1]{Ag} I.~Agol, \textit{Criteria for virtual fibering}, J.~Topol.~\textbf{1} (2008), 269--284.

\bibitem[Ag2]{Agol-VHC} I.~Agol, 
\textit{The virtual Haken conjecture}, with an appendix by I.~Agol, D.~Groves, J.~Manning, preprint, 2012,
\texttt{arXiv:1204.2810}.

\bibitem[BS]{BS} S.~V.~Buyalo, P.~V.~Svetlov, \textit{Topological and geometric properties of graph manifolds}, 
Algebra i Analiz \textbf{16} (2004), no.~2, 3--68.

\bibitem[DJ]{DJ} M.~Davis, T.~Januszkiewicz, \textit{Right-angled Artin groups are commensurable with right-angled 
Coxeter groups}, J.~Pure Appl.~Algebra \textbf{153} (2000), no.~3, 229--235.

\bibitem[Gr]{Gr} M.~Gromov, \emph{Hyperbolic groups}, Essays in Group Theory, Math.~Sci.~Res.~Inst.~Publ.~
8, Springer, NY, 1987, pp. 75--263.

\bibitem[GW]{GW} D.~Gromoll, J.~A.~Wolf, \textit{Some relations between the metric structure 
and the algebraic structure of the fundamental group in manifolds of nonpositive curvature},
Bull.~Amer.~Math.~Soc.~\textbf{77} (1971), 545--552.

\bibitem[HW]{HW} F.~Haglund, D.~T.~Wise, \textit{Special cube complexes}, Geom.~Funct.~Anal.~\textbf{17} (2008),
1551--1620.

\bibitem[HsW]{HsW} T.~Hsu, D.~T.~Wise, \textit{On linear and residual properties of graph products},
Michigan Math.~J.~\textbf{46} (1999), no.~2, 251--259.

\bibitem[Ja]{Ja} W.~Jaco, \textit{Lectures on Three-Manifold Topology},
CBMS Regional Conference Series in Mathematics, 43. 
American Mathematical Society, Providence, RI, 1980. 

\bibitem[KL]{KL} M.~Kapovich, B.~Leeb, \textit{Actions of discrete groups on nonpositively curved spaces},
Math.~Ann.~\textbf{306} (1996), no.~2, 341--352.

\bibitem[Le]{Le} B.~Leeb, \textit{$3$-manifolds with(out) metrics of nonpositive curvature}, Invent.~Math.~
\textbf{122} (1995), 277--289.

\bibitem[LW]{LW} J.~Luecke, Y.-Q.~Wu, \textit{Relative Euler number and finite covers of graph manifolds},
Geometric Topology (Athens, GA, 1993), AMS/IP Stud.~Adv.~Math., vol.~2.1, American Mathematical Society, Provi-
dence, RI, 1997, pp.~80--103.

\bibitem[Ne]{Ne} W.~D.~Neumann, \textit{Immersed and virtually embedded $\pi_1$-injective surfaces in graph manifold},
Algebr.~Geom.~Topol. \textbf{1} (2001), 411--426.

\bibitem[Ni]{Ni} B.~Nica, \textit{Cubulating spaces with walls}, Algebr.~Geom.~Topol.~\textbf{4} (2004), 297--309.

\bibitem[PW1]{PW-graph} P.~Przytycki, D.~T.~Wise, \textit{Graph manifolds with boundary are virtually special}, in preparation.

\bibitem[PW2]{PW-mixed} P.~Przytycki, D.~T.~Wise, \textit{Mixed 3-manifolds are virtually special}, preprint, 2012, \texttt{arXiv:1205.6742}.

\bibitem[RW]{RW} J.~H.~Rubinstein, S.-C. Wang, \textit{$\pi_1$-injective surfaces 
in graph-manifolds}, Comment.~Math.~Helv.~\textbf{73} (1998), 499--515.

\bibitem[Sc]{Sc-LERF} P.~Scott, \textit{Subgroups of surface groups are almost geometric}, J.~London Math.~Soc.~(2)
\textbf{17} (1978), no.~3, 555--565.

\bibitem[Sv]{Sv} P.~V.~Svetlov, \textit{Graph-manifolds of nonpositive curvature are virtually fibered over a circle}, (Russian),
Algebra i Analiz \textbf{14} (2002), no.~5, 188--201; English translation in 
St.~Petersburg Math.~J.~\textbf{14} (2003), no.~5, 847--856.

\bibitem[Th1]{Th-VFC} W.~P.~Thurston, \textit{Three dimensional manifolds, Kleinian groups and hyperbolic geometry},
Bull.~Amer.~Math.~Soc.~(N.S.) \textbf{6} (1982), no.~3, 357--381.

\bibitem[Th2]{Th-NT} W.~P.~Thurston, \textit{On the geometry and dynamics of diffeomorphisms of surfaces},
Bull.~Amer.~Math.~Soc.~\textbf{19} (1988), no.~2, 417--438.

\bibitem[WY]{WY} S.-C.~Wang, F.-C.~Yu, \textit{Graph-manifolds with non-empty boundary are covered by a surface
bundle over the circle}, Math.~Proc.~Camb.~Phil.~Soc.~\textbf{122} (1997), 447--455.

\bibitem[Wi]{Wi} D.~T.~Wise, \textit{Research announcement: the structure of groups with a quasiconvex hierarchy},
Electron.~Res.~Announc.~Math.~Sci.~\textbf{16} (2009), 44--55.

\bibitem[Ya]{Ya} S.-T.~Yau, \textit{On the fundamental group of compact manifolds of non-positive curvature},
Ann.~of Math.~(2) \textbf{93} (1971), no.~3, 579--585. 


\end{thebibliography}

\end{document}